\documentclass{amsart}
\usepackage{amsmath,amsfonts,amssymb,amscd,verbatim,multicol}
\usepackage[arrow,matrix,cmtip,curve]{xy}
\usepackage{lscape}
\usepackage{mathrsfs}  
\usepackage{tikz,tikz-cd}
\usepackage{hyperref}

\usepackage{fullpage}
\allowdisplaybreaks

\numberwithin{equation}{section}
\theoremstyle{definition}
\newtheorem{theorem}[equation]{Theorem}

\newtheorem{corollary}[equation]{Corollary} 
\newtheorem{definition}[equation]{Definition} 
\newtheorem{example}[equation]{Example}
\newtheorem{lemma}[equation]{Lemma}
\newtheorem{proposition}[equation]{Proposition}

\newtheorem{remark}[equation]{Remark}

\newtheorem{thmintro}{Theorem}

\DeclareMathOperator\Aut{Aut}

\DeclareMathOperator\Ext{Ext}

\DeclareMathOperator\gr{gr}
\DeclareMathOperator\GKdim{GKdim}
\DeclareMathOperator\gldim{gldim}

\DeclareMathOperator\pdim{pdim}

\DeclareMathOperator\tot{tot}

\DeclareMathOperator\soc{soc}

\renewcommand\int{\mathrm{int}}

\newcommand\inv{^{-1}}

\newcommand\iso{\cong}
\newcommand\tensor{\otimes}
\newcommand\kk{\Bbbk}

\newcommand{\op}{\mathrm{op}}

\newcommand\cA{\mathcal A}
\newcommand\cB{\mathcal B}
\newcommand\cC{\mathcal C}
\newcommand\cD{\mathcal D}
\newcommand\cF{\mathcal F}
\newcommand\cH{\mathcal H}

\newcommand\cR{\mathcal R}

\newcommand\NN{\mathbb N}

\newcommand\ZZ{\mathbb Z}

\newcommand\bp{\mathbf p}
\newcommand\bq{\mathbf q}

\newcommand\fs{\mathfrak s}
\newcommand\ft{\mathfrak t}

\newcommand\dblQ{\overline Q}
\newcommand\tldQ{\widetilde Q}
\newcommand\hatQ{\widehat Q}
\newcommand\fQ{\mathfrak Q}

\begin{document}

\title{Normal extensions of twisted graded Calabi--Yau algebras}

\author[Gaddis]{Jason Gaddis}
\author[Keeler]{Dennis Keeler}
\address{Miami University, Department of Mathematics, Oxford, Ohio 45056} 
\email{gaddisj@miamioh.edu,keelerds@miamioh.edu}


\begin{abstract}
We discuss three families of twisted graded Calabi--Yau algebras of dimension three that arise through the process of normal extensions. In addition to the twisted Calabi--Yau condition, we prove that these algebras are noetherian piecewise domains under a non-degeneracy condition. We show that these algebras can also be constructed as certain iterated skew polynomial rings. Finally, we discuss how these algebras shed light on possible types of twisted graded Calabi--Yau algebras of dimension four.
\end{abstract}

\maketitle

Throughout, let $\kk$ be an algebraically closed field of characteristic zero. Let $\kk^\times$ denote the group of units of $\kk$. All algebras will be assumed to be $\kk$-algebras. Let $\delta_{ij}$ denote the Kronecker-delta function.

In \cite{GLNW,GR1}, the authors produced a classification of quivers that may support twisted-graded Calabi--Yau algebras of (global) dimension three and finite Gelfand--Kirillov dimension. In many cases, it is possible to produce an explicit example of such an algebra on a given quiver. The most common method is to form a skew group ring $A\#\kk G$ where $A$ is an Artin--Schelter (AS) regular algebra of dimension three and $G$ is a finite subgroup of $\Aut_{\gr}(A)$. Other useful methods include Ore extensions and graded (Zhang) twists. 

In this paper, we show that the method of normal extensions can be used to produce several families of twisted graded Calabi--Yau algebras of dimension three. 
By the work of Zhou and Shen, the twisted Calabi--Yau property lifts to normal extensions \cite{ZSL}. Our techniques follow the work of Chirvasitu, Kanda, and Smith \cite{CKS}, wherein the authors studied normal extensions of Artin--Schelter regular algebras. We are inspired by the following example.

\begin{example}\label{ex.downup}
Let $q \in \kk^\times$ and let $\omega = xy-qyx \in \kk\langle x,y \rangle$. Then $A=\kk\langle x,y \rangle/(\omega)$ is an Artin--Schelter regular algebra of global dimension two. In particular, it is a \emph{quantum plane}. For $p_1,p_2 \in \kk^\times$, define
\begin{align}\label{eq.Bdim1}
    B = \frac{\kk\langle x,y \rangle}{(x\Omega-p_1\Omega x, y\omega-p_2\omega y)}.
\end{align}
If $p_2=p_1\inv$, then $B$ is Artin--Schelter regular of global dimension three. In fact, the resulting algebra is a \emph{(graded) down-up algebra} and all such down-up algebras that are Artin--Schelter regular can be realized this way. Moreover, $\omega$ is normal in $B$ and $A \iso B/(\omega)$. For a generic choice of $p_1,p_2$, $B$ will not be Artin--Schelter regular.

For $p,q \in \kk^\times$, define the \emph{quantum Heisenberg algebra} as
\[ H_{p,q} = \frac{\kk\langle x,y,z \rangle}{(xz-pxz, yz-p\inv zy, xy-qyx-z)}.\]
The algebra $H_{1,1}$ is the enveloping algebra of the Heisenberg Lie algebra. In general, $H_{p,q}$ is Artin--Schelter regular of global dimension three with grading given by $\deg(x)=\deg(y)=1$ and $\deg(z)=2$. It can be realized as an iterated Ore extension and so $H_{p,q}$ is a noetherian domain. The element $z$ is normal and regular, but also superfluous. Substituting $xy-qyx$ for $z$ in the other two relations gives the algebra $B$ in \eqref{eq.Bdim1} with $p_1=p$ and $p_2=p\inv$.
\end{example}

In earlier work \cite{GK1}, the authors presented a family of algebras that generalize the down-up algebras, which were originally defined by Benkart and Roby \cite{BRdu,BRduADD}.

\begin{definition}\label{defn.qdu}
Let $\dblQ$ be the quiver: 
\begin{equation}\label{eq.dblA} 
\begin{tikzcd}
    &         &   	& e_0 \arrow[llld, "u_0"] \arrow[rrrd, "d_{n-1}", shift left=2] 	&           &		&          \\
e_1 \arrow[rr, "u_1"] \arrow[rrru, "d_0", shift left=2] 	&        	&                   	e_2 \arrow[r, "u_2"]	\arrow[ll, "d_1", shift left=2]			 		& \cdots   \arrow[r, "u_{n-3}"]  \arrow[l, "d_2", shift left=2] &	e_{n-2} \arrow[rr, "u_{n-2}"]	\arrow[l, "d_{n-3}", shift left=2] &           	& e_{n-1} \arrow[lllu, "u_{n-1}"] \arrow[ll, "d_{n-2}", shift left=2]
\end{tikzcd}
\end{equation}
Fix $n \in \ZZ_+$ and let $\alpha,\beta,\gamma \in \kk^n$.
The corresponding \emph{quiver down-up algebra of type $A$} is defined as the quotient of $\kk\dblQ$ by the relations
\begin{equation}\label{eq:relationsA}
\begin{aligned}
	d_{i-1}u_{i-1}u_i &= \alpha_i u_id_iu_i + \beta_i u_iu_{i+1}d_{i+1} +\gamma_i u_i \\
	d_id_{i-1}u_{i-1} &= \alpha_i d_iu_id_i + \beta_i u_{i+1}d_{i+1}d_i +\gamma_i d_i,
\end{aligned}
\end{equation}
for $i\in \dblQ_0$. We denote this algebra by $\cH_n(\alpha,\beta,\gamma)$. If $\gamma_i=0$ for all $i \in \dblQ_0$, then we write $\cH_n(\alpha,\beta)$.
\end{definition}

Throughout, we interpret paths on the quiver \eqref{eq.dblA} and others modulo $n$. The reason for the notation $\dblQ$ is that the quiver \eqref{eq.dblA} is the double of the extended Dynkin quiver of type $\widetilde{A_{n-1}}$, see \eqref{eq.typeA}.

On one vertex, the $\cH_n(\alpha,\beta,\gamma)$ recover the algebras from Example \ref{ex.downup}. In \cite{GK1}, the authors showed that the $\cH_n(\alpha,\beta,\gamma)$ are twisted Calabi--Yau if and only if $\beta_i\neq 0$ for all $i$. In this work, we primarily focus on the $\cH_n(\alpha,\beta)$, which are $\NN$-graded, and we show that certain ones appear as normal extensions of twisted preprojective algebras of type A. This gives an alternate way to realize certain results in \cite{GK1}. The construction we give is broad enough to apply to other large families of twisted graded Calabi--Yau algebras that we discuss now.

\begin{definition}\label{defn.qDU_B}
For $n \geq 2$, set $\tldQ$ to be the quiver:
\begin{equation}\label{eq.Bn}
\begin{tikzcd}
    &   & e_0 \arrow[out=120,in=60,loop,"b_0"]\arrow[lld,"a_0"']   &   &   \\
e_1 \arrow[r,"a_1"'] \arrow[out=240,in=300,loop,swap,"b_1"] & 
e_2 \arrow[r,"a_2"'] \arrow[out=240,in=300,loop,swap,"b_2"] & 
\cdots \arrow[r,"a_{n-3}"'] & 
e_{n-2} \arrow[r,"a_{n-2}"'] \arrow[out=240,in=300,loop,swap,"b_{n-2}"] & e_{n-1} \arrow[llu,"a_{n-1}"'] \arrow[out=240,in=300,loop,swap,"b_{n-1}"]
\end{tikzcd}
\end{equation}
For $\alpha,\beta\in \kk$, the corresponding \emph{quiver down-up algebra of type B} is defined as the quotient of $\kk\tldQ$ by the relations
\begin{align*}
	b_ia_ia_{i+1} &= \alpha a_ib_{i+1}a_{i+1} + \beta a_ia_{i+1}b_{i+2} \\
    b_i^2a_i &= \alpha b_ia_ib_{i+1} + \beta a_ib_{i+1}^2,
\end{align*}
for $i \in \tldQ_0$. We denote this algebra by $\cB_n(\alpha,\beta)$.
\end{definition}

\begin{definition}\label{defn.qDU_C}
For $n \geq 2$, set $\hatQ$ to be the quiver:
\begin{equation}\label{eq.typeC}
\begin{tikzcd}
& & 
e_0 \arrow[lld,bend right,"b_0"'] \arrow[lld,"a_0"'] & & \\
e_1 \arrow[r,bend right,swap,"b_1"] \arrow[r,"a_1"'] & 
e_2 \arrow[r,bend right,swap,"b_2"] \arrow[r,"a_2"'] & 
\cdots \arrow[r,"a_{n-3}"'] \arrow[r,bend right,swap,"b_{n-3}"] & 
e_{n-2} \arrow[r,"a_{n-2}"'] \arrow[r,bend right,swap,"b_{n-2}"] & e_{n-1} \arrow[llu,"a_{n-1}"'] \arrow[llu,bend right,swap,"b_{n-1}"]
\end{tikzcd}
\end{equation}
For $\alpha,\beta\in \kk$, the corresponding \emph{quiver down-up algebra of type C} is defined as the quotient of $\kk\hatQ$ by the relations
\begin{align*}
    b_ia_{i+1}a_{i+2} &= \alpha a_ib_{i+1}a_{i+2} + \beta a_ia_{i+1}b_{i+2} \\
	b_ib_{i+1}a_{i+2} &= \alpha b_ia_{i+1}b_{i+2} + \beta a_ib_{i+1}b_{i+2}
\end{align*}
for $i\in \hatQ_0$. We denote this algebra by $\cC_n(\alpha,\beta)$.
\end{definition}

\begin{thmintro}[Theorems \ref{thm.typeA}, \ref{thm.typeB}, and \ref{thm.typeC}]
\label{thm.down-up}
Fix $n \in \ZZ_+$ and let $\alpha,\beta \in \kk$. Let $\cD$ be one of $\cH_n(\alpha,\beta)$, $\cB_n(\alpha,\beta)$, or $\cC_n(\alpha,\beta)$.
The following are equivalent:
\begin{enumerate}
\item $\beta \neq 0$,
\item \label{Dpwd} $\cD$ is a piecewise domain,
\item \label{Dnoeth} $\cD$ is right (or left) noetherian.
\end{enumerate}
Moreover, if any of the above conditions hold, then $\cD$ is twisted graded Calabi--Yau of dimension three. 
\end{thmintro}

Theorem \ref{thm.down-up} is proved more generally for the $\cH_n(\alpha,\beta,\gamma)$ in \cite{GK1}. The goal in this work is to study normal extensions in the context of quotients of path algebras where these methods may be applied more broadly to the B and C types.

In Section \ref{sec.background}, we provide necessary background on quivers, path algebras, and the twisted Calabi--Yau condition. We also introduce three families of twisted graded Calabi--Yau algebras of dimension two. The normal extensions of these will give the algebras introduced above. In Section \ref{sec.lift} we provide the general technology for lifting the twisted Calabi--Yau property through normal extensions. In Sections \ref{sec.typeA} and \ref{sec.typeBC}, we consider each of the three families and prove Theorem \ref{thm.down-up}.

Section \ref{sec.ore} presents an alternate proof from the perspective of Ore extensions. Let $Q=\widetilde{A_{n-1}}$ be the extended Dynkin quiver of type A:
\begin{equation}\label{eq.typeA}
\begin{tikzcd}
    & & e_0 \arrow[lld,"a_0"'] &    & \\
e_1 \arrow[r,"a_1"'] & 
e_2 \arrow[r,"a_2"'] & 
\cdots \arrow[r,"a_{n-3}"'] & 
e_{n-2} \arrow[r,"a_{n-2}"'] & 
e_{n-1} \arrow[llu,"a_{n-1}"']
\end{tikzcd}
\end{equation}
Let $n \geq 2$ and let $\sigma$ be a graded automorphism of $\kk Q$. Then there are scalars $q_i \in \kk^\times$ such that $\sigma(e_i) = e_{\sigma(i)}$ and $\sigma(a_i)=q_i a_{\sigma(i)}$ for all $i$. Set $A=\kk Q[x;\sigma]$ and let $\fQ$ be the underlying quiver of $A$ so that $A = \kk\fQ/(\omega)$ where
\[ \omega = \sum_{i \in Q_0} x_ia_i - q_i a_{\sigma(i)}x_{i+1}.\]

\begin{thmintro}[Theorem {\ref{thm.ore}}]
\label{thmintro.ore}
Let $A$ be as above with $n \geq 3$ and suppose that $q_i = q_{\sigma(i-1)}$ for all $i \in Q_0$. Define $\phi \in \Aut_{\gr}(A)$ 
\[ 
\phi(e_i) = e_{\sigma(i-1)}, \quad
\phi(a_i) = p_i a_{\sigma(i-1)}, \quad
\phi(x_i) = p_i\inv x_{\sigma(i-1)},
\] 
where $p_i \in \kk^\times$ and $p_{\sigma(i)}=p_{i+1}$ for all $i \in Q_0$. Then the normal extension 
\[ \cD(\fQ,\omega,\phi) = \kk Q/(\omega a - \phi(a)\omega : a \in \fQ_1)\]
is twisted Calabi--Yau of dimension three.
\end{thmintro}

Theorem \ref{thmintro.ore} is then applied to each of the types (A, B, and C). See Proposition \ref{prop.ore}.

Finally, in Section \ref{sec.dim4}, we consider examples of taking normal extensions of twisted graded Calabi--Yau algebras of dimension three to obtain one of dimension four. There are three ``types" of Artin--Schelter regular algebras corresponding to resolutions of the trivial module and the form of the Hilbert series. We show that one can obtain analogues of these for quivers on many vertices through normal extensions. In particular, we describe a family of twisted graded Calabi--Yau algebras of dimension four that arise as normal extensions of a certain quiver down-up algebra.

\subsection*{Acknowledgements}
Gaddis is partially supported by an AMS–Simons Research Enhancement Grant for PUI Faculty.

\section{Background}
\label{sec.background}

Here we review important concepts on path algebras and the twisted Calabi--Yau condition.

\subsection{Graded algebras and Hilbert series}

An algebra $A$ is \emph{$\NN$-graded} if there is a vector space decomposition $A=\bigoplus_{k \in \NN} A_k$ such that $A_k \cdot A_\ell \subset A_{k+\ell}$. Throughout, \emph{graded} will mean $\NN$-graded. A graded algebra $A$ is \emph{locally finite} if $\dim_{\kk} A_k < \infty$ for all $k$. 

\subsection{Quivers and path algebras}

A \emph{quiver} $Q$ is a tuple $(Q_0,Q_1)$ where $Q_0$ is a set of vertices and $Q_1$ a set of edges along with maps $\fs,\ft: Q_1 \to Q_0$. For each $a \in Q_1$, $\fs(a)$ is the \emph{source} and $\ft(a)$ is the \emph{target}. Throughout, we assume that $Q_0$ and $Q_1$ are finite. A $\kk$-linear basis of the \emph{path algebra} $\kk Q$ consists of \emph{paths} of arrows $p=a_0a_1\hdots a_n$ with $\ft(a_i)=\fs(a_{i+1})$ for $i = 0,\hdots,n-1$. The maps $\fs,\ft$ extend naturally to the set of paths with $\fs(p)=\fs(a_0)$ and $\ft(p)=\ft(a_n)$. Given paths $p$ and $q$, multiplication is defined as $p \cdot q = pq$ (concatenation) if $\ft(p)=\fs(q)$ and $p \cdot q = 0$ otherwise. At each vertex $k \in Q_0$ there is a trivial path $e_k$ which has the properties $e_k p = p$ (resp. $pe_k=p$) if $\fs(p)=k$ (resp. $\ft(p)=k$) and $0$ otherwise. There is a filtration on $Q$, which extends to $\kk Q$, by path length. An element $f \in \kk Q$ is \emph{homogeneous} if $f$ is the sum of elements of the same path length. An ideal $I$ in $\kk Q$ is \emph{homogeneous} if it is generated by homogeneous elements.

The \emph{adjacency matrix} $M_Q\in M_n(\NN)$ of a quiver $Q$ with $|Q_0|=n$ where $(M_Q)_{ij}$ denotes the number of arrows with source $i$ and target $j$. Let $I$ be a homogeneous ideal in $\kk Q$ and set $A=\kk Q/I$. For each $k \in \NN$, let $H_k \in M_n(\NN)$ where $(H_k)_{ij}=\dim(e_iAe_j)_k$. If $I \subset (\kk Q)^2$, then $H_1 = M_Q$. The \emph{matrix-valued Hilbert series} of $A$ is defined as 
\[ h_A(t) = \sum_{k=0}^\infty H_k t^k.\]
The \emph{total Hilbert series} of $A$, $h_A^{\tot}(t)$, is the infinite series in which the coefficient of $t^k$ is the sum of the entries of $H_k$.

\subsection{Twisted CY algebras}

The \emph{enveloping algebra} of an algebra $A$ is 
defined as $A^e = A \tensor_\kk A^{\text{op}}$. 

\begin{definition}\label{defn.twCY}
An algebra $A$ is \emph{homologically smooth} if $A$ has a finite length resolution by finitely generated projective $A^e$-modules. If, further, there exists an invertible $(A,A)$-bimodule $U$ such that there are right $A^e$-module isomorphisms $\Ext_{A^e}^i(A,A^e) \iso \delta_{id} U$, then $A$ is \emph{twisted Calabi--Yau} (of dimension $d$). If $U = {}^1A^\mu$ for an automorphism $\mu$ of $A$, then $\mu$ is called the \emph{Nakayama automorphism} of $A$
\end{definition}

We use $\mu^A$ to denote the Nakayama automorphism of the algebra $A$ and we will omit the superscript when the algebra is clear. The Nakayama is unique up to composition with an inner automorphism \cite[Lemma 1.7(d)]{RRZ}.

One can also make appropriate adjustments to the above definition to obtain the definition of \emph{graded} twisted Calabi--Yau algebras. By \cite[Theorem 4.2]{RR2}, a graded algebra $A$ is graded twisted Calabi--Yau if and only if it is twisted Calabi--Yau. In this setting, the Nakayama automorphism $\mu^A$ is necessarily graded, and we use the notation $\mu_0^A$ to denote the restriction of $\mu^A$ to $A_0$.

A connected graded algebra $A$ is twisted Calabi--Yau if and only if it is Artin--Schelter regular \cite[Lemma 1.2]{RRZ}. Note that in the reference and others, ``twisted Calabi--Yau" is called ``skew Calabi--Yau."

For a (graded) algebra $A$, let $J(A)$ denote its (graded) Jacobson radical. If $A$ is locally finite and graded, then the irrelevant ideal $A_{\geq 1}$ is contained in $J(A)$. Let $S=A/J(A)=A_0/J(A_0)$. Then $S$ is a finite-dimensional semisimple algebra.

\begin{definition}\label{defn.genAS}
Let $A$ be a locally finite graded $\kk$-algebra, with $S=A/J(A)$. Then $A$ is \emph{generalized Artin--Schelter (AS) regular of dimension $d$} if $A$ has graded global dimension $d$ and there is a $\kk$-central invertible $(S,S)$-bimodule $V$ such that $\Ext^i(S,A) = \delta_{id} V$ as $(S,S)$-bimodules.
\end{definition}

When $A$ is a locally finite graded algebra and $S$ is separable, then twisted Calabi--Yau and generalized Artin--Schelter regular are equivalent \cite[Theorem 1.5]{RR2}. Moreover, in this setting, $\gldim_l(A) = \pdim(_A S) = \pdim(S_A) = \gldim_r(A)$ \cite[Proposition 3.18]{RR2}. This is the setting we work in and so we simply denote all of these values by $\dim(A)$.

\subsection{Twisted graded Calabi--Yau algebras of dimension two}

The discussion below is a restricted version of a construction from \cite{RR1}. In particular, we only consider the case where all arrows are weighted in degree 1, whereas the authors in \cite{RR1} consider the case that arrows may be given arbitrary weights.

Let $Q$ be a (finite) quiver with $|Q_0|=n$. As above, we weight all arrows in degree 1 so that $\kk Q$ is a graded algebra (generated in degree 1). Set $V=\kk Q_1$ and let $\tau:V \to V$  be a bijective $\kk$-linear map. That is, $\tau \in GL(V)$. Suppose that $\tau$ induces a permutation $\pi$ of $Q_0$ such that $\tau(e_i V e_j) \subset e_{\pi\inv(j)} V e_i$ for all $i,j,d$. Set $\omega= \sum_{x \in Q_1} \tau(x)x$. Then
\[ \cA(Q,\tau) = \kk Q/(\omega) = \kk Q/(\omega_1,\hdots,\omega_n)\]
where $\omega_i = \omega e_i = e_{\pi\inv(i)} \omega e_i$. Let $M$ denote the adjacency matrix of $Q$, and $P$ the permutation matrix induced by $\pi=\mu_0^\cA$ (so $P_{ij} = \delta_{\pi(i)j}$). By \cite[Lemma 7.6]{RR1}, $\cA=\cA(Q,\tau)$ is twisted Calabi--Yau if and only if $h_{\cA}(t) = (p(t))\inv$ where $p(t)=I-Mt+Pt^2$. 

If we assume that $\GKdim(\cA) < \infty$, then by \cite[Theorem 7.8]{RR1}, the spectral radius of $M$ satisfies 
\[ \rho(M)= \max\{ |\lambda| \mid \text{$\lambda$ is an eigenvalue of $M$}\} = 2.\] 
It follows further that $\GKdim(\cA)=2$ and $\cA$ is noetherian \cite[Theorem 6.6]{RR2}. We now give several examples of twisted graded Calabi--Yau algebras.

\begin{definition}\label{defn.typeA}
Let $n \in \ZZ_+$ and let $\dblQ$ be as in \eqref{eq.dblA}. For $\bq \in (\kk^\times)^n$, define
\[ A_n(\bq) := \kk\dblQ/ (d_iu_i - q_i u_{i+1}d_{i+1} \quad\text{for $i\in\dblQ_0$}).\]
In case there is some $q \in \kk^\times$ such that $q_i=q$ for all $i$, we denote $A_n(\bq)$ by $A_n(q)$.
\end{definition}

The algebras $A_n(\bq)$ are twisted Calabi--Yau by \cite[Lemma 2.2]{GZiso}. They are Calabi--Yau precisely when $q_i=1$ for all $i$, in which case $A_n(1)$ is (isomorphic to) a preprojective algebra of type $A$. Let $M$ be the adjacency matrix of $\dblQ$. Then
\begin{align}\label{eq.Ahilb} 
    h_{A_n(\bq)} = (I-Mt+It^2)\inv
    \quad\text{and}\quad
    h_{A_n(\bq)}^{\tot} = n(1-t)^{-2}.
\end{align}

\begin{definition}\label{defn.typeB}
Let $n \geq 2$ and let $\tldQ$ be as in \eqref{eq.Bn}. For $\bq \in (\kk^\times)^n$,
define 
\[ B_n(\bq) = \kk\tldQ/(b_ia_i - q_i a_i b_{i+1} \quad 
    \text{for $i\in\tldQ_0$}).\]
In case there is some $q \in \kk^\times$ such that $q_i=q$ for all $i$, we denote $B_n(\bq)$ by $B_n(q)$.
\end{definition}

Let $B=B_n(\bq)$. Then $B$ is twisted Calabi--Yau of dimension two and the Nakayama automorphism satisfies $\mu_0^B(e_i)=e_{i+1}$ \cite[Lemma 2.3]{GZiso}. The permutation matrix $P^B$ corresponding to $\mu_0^B$ satisfies 
\begin{align}\label{eq.PtypeB}
P_{ij}^B = \begin{cases} 
    1 & i=j-1 \\
    0 & \text{otherwise}.
\end{cases}
\end{align}
The adjacency matrix of $\tldQ$ satisfies $M=I+P$ and so
\[ 
    h_{B_n(\bq)} 
        = (I-Mt+P^Bt^2)\inv 
        = \sum_{k=0}^\infty H_k t^k \quad\text{where}\quad
    H_k = \sum_{i=0}^k (P^B)^i,
\] 
and $h_{B_n(\bq)}^{\tot}=n(1-t)^{-2}$.

\begin{definition}\label{defn.typeC}
Let $n \geq 2$ and let $\hatQ$ be as in \eqref{eq.typeC}. For $\bq \in (\kk^\times)^n$, define 
\[ C_n(\bq) = \kk\hatQ/(b_ia_{i+1}-q_ia_ib_{i+1} \quad \text{for $i\in\hatQ_0$}).\]
In case there is some $q \in \kk^\times$ such that $q_i=q$ for all $i$, we denote $C_n(\bq)$ by $C_n(q)$.
\end{definition}

Set
\begin{align}\label{eq.PtypeC}
P_{ij}^C = \begin{cases} 
    1 & i=j-2 \\
    0 & \text{otherwise}.
\end{cases}
\end{align}
The following should be compared to \cite[Lemmas 2.2, 2.3]{GZiso}.

\begin{lemma}\label{lem.typeC}
Fix $n \geq 2$, let $\hatQ$ be the quiver from \eqref{eq.typeC}, and let $M$ be the adjacency matrix of $\hatQ$. For $\bq \in (\kk^\times)^n$, the algebras $C_n(\bq)$ are twisted Calabi--Yau of dimension two and the Nakayama automorphism $\mu^C$ satisfies $\mu_0^C(e_i) = e_{i+2}$.
\end{lemma}
\begin{proof}
Set $C=C_n(\bq)$. Define $\tau$ by $\tau(a_i) = b_{i-1}$ and $\tau(b_i) = -q_{i-1} a_{i-1}$ extended linearly so that $\tau:V \to \kk\hatQ$ is an injective $\kk$-linear graded map. It is clear that $W=\tau(V)$ is again an arrow space for $\kk\hatQ$. So, $C_n(\bq) \iso \cA(\hatQ,\tau)$. Let $\omega = \sum_{x \in \hatQ_1} \tau(x)x$ so that 
\[ \omega_{i+2} = e_i \omega e_{i+2} = \tau(a_{i+1})a_{i+1} + \tau(b_{i+1})b_{i+1}
	= b_ia_{i+1} - q_ia_ib_{i+1}.\]
In this case, $\mu(e_i)=e_{i+2}$ and so the permutation matrix corresponding to $\mu_0$ is $P^C$ defined in \eqref{eq.PtypeC}. Set $N$ to be the $n \times n$ matrix with $N_{j-1,j}=1$ and $N_{ij}=0$ for $i \neq j-1$ so that $N^2=P^C$. Note that the adjacency matrix for $\hatQ$ is $2N$.

It is straightforward (see Lemma~\ref{lem.Cbasis}) that a $\kk$-basis for $C$ consists of paths of the form
\[ a_ia_{i+1}\cdots a_{i+k-1}b_{i+k}b_{i+k+1} \cdots b_{i+k+\ell-1},\]
for $k,\ell \in \NN$ and $i \in \hatQ_0$. That is, there is a bijection between basis paths and monomials in $\kk[x,y]$. Therefore, the matrix-valued Hilbert series for $C$ satisfies $h_C(t) = \sum H_k t^k$, where $H_k = (k+1)N^k$, and $h_C^{\tot}=n(1-t)^2$.

By \cite[Lemma 7.6]{RR1}, $C=\cA(\hatQ,\tau)$ is twisted Calabi--Yau of dimension two if and only if $h_C(t) = (p(t))\inv$ where $p(t)=I-Mt+Pt^2$. That is, we must show that the the coefficients of $h_C(t)$ satisfy the recursive formula $H_0=I$, $H_1=M$, and $H_{k+1} = MH_k - P^CH_{k-1}$. Suppose this holds for some $k$. Then
\[ MH_k - P^CH_{k-1} = (2N)( (k+1) N^k) - N^2 ( kN^{k-1}) = (k+2)N^{k+1} = H_{k+1},\]
as claimed.
\end{proof}

The algebra $A_2(q_0,q_1)$ is isomorphic to $C_2(q_1,q_0)$ via the map $e_i \mapsto e_i$, $a_i \mapsto u_i$, and $b_i \mapsto d_{i-1}$ for $i=0,1$.

\subsection{Twisted graded Calabi--Yau algebras of dimension three}

To describe the case of dimension three, we first review some material on superpotentials.

Let $Q$ be a quiver and let $\eta$ be an automorphism of $\kk Q$. An \emph{$\eta$-twisted superpotential} of $\kk Q$ is a linear combination of paths of length $d$ which is invariant under the linear map sending the path $\alpha_1 \alpha_2 \cdots \alpha_d$ to $(-1)^{d+1} \eta(\alpha_d)\alpha_1 \cdots \alpha_{d-1}$. For each $\alpha \in Q_0$, there is a (linear) derivation map $\delta_\alpha$ sending the path $\alpha_1 \alpha_2 \cdots \alpha_d$ to $\alpha_2 \cdots \alpha_d$ if $\alpha=\alpha_1$ and $0$ otherwise. Given an $\eta$-twisted superpotential $\omega$, the corresponding \emph{derivation-quotient algebra (of degree one)} is $\kk Q/I$ where $I = (\partial_a \omega : a \in Q_1)$. Hence, $I$ is generated by relations of the same degree.

In the name of efficiency, we will often use shorthand to write superpotentials. Let $p=\alpha_1 \alpha_2 \cdots \alpha_d$ be a summand of a superpotential $\omega$. Let $S$ denote the set of signed $\eta$-twists of $p$. Then the \emph{compact form} of $p$ is $[p]=\sum_{s \in S} s$.

Let $A=\kk Q/I$ be graded twisted Calabi--Yau of dimension three. By \cite[Theorem 6.8, Remark 6.9]{BSW}, $A$ is a derivation-quotient algebra for some $\eta$-twisted superpotential $\omega$. Moreover, the Nakayama automorphism of $A$ is $\mu^A = (-1)^{d+1}\eta\inv$ \cite[Theorem 6.8]{BSW}. Note that our convention on composing paths is opposite from that reference, hence the inverse. 

Let $M$ be the adjacency matrix for $Q$ and $P$ the permutation matrix corresponding to $\mu_0^A$. Let $d$ denote the degree of $\omega$. We call $(M,P,d)$ the \emph{type} of $A$. We also say that $M$ \emph{supports} a twisted Calabi--Yau algebra.

\begin{example}[{\cite[Lemma 3.2 and Theorem 3.6]{GK1}}]
Let $\dblQ$ be as in \eqref{eq.dblA}. Suppose $\beta_i \neq 0$ for all $i\in \dblQ_0$. Define $\eta \in \Aut_{\gr}(\kk\dblQ)$ by $\eta(u_i)=\beta_{i-1} u_i$ and $\eta(d_i)=\beta_{i-1} d_i$ for all $i\in Q_0$. Then the algebra $\cH = \cH_n(\alpha,\beta)$ is the derivation-quotient algebra on $\kk Q$ with $\eta$-twisted superpotential
\[ \Omega = \sum_{i\in Q_0} [d_i d_{i-1} u_{i-1} u_i] - \alpha_i [d_i u_i d_i u_i].\]
The Nakayama automorphism $\mu^{\cH}$ is then given by $\mu^{\cH}(u_i)=-\beta_{i-1}\inv u_i$ and $\mu^{\cH}(d_i)=-\beta_{i-1}\inv d_i$.
\end{example}

We will also assume that $\GKdim(A)=3$, though in some cases it is sufficient only to assume that $\GKdim(A) < \infty$. By \cite{RR1}, $M$ and $P$ commute and $h_A(t) = (p(t))\inv$ where
\begin{align}\label{eq.dim3_hilb}
    p(t) = I - Mt + PM^T t^{d-1} - Pt^d.
\end{align}
Every root of $\det(p(t))$ is a root of unity and $\det(p(t))$ vanishes at $t=1$ to order at least 3. It is conjectured that $d=3$ or $d=4$, and it is known that this holds when $M$ is normal ($MM^T=M^TM$). It is an open question, in general, when a matrix $M$ supports a graded twisted Calabi--Yau algebra $A=\kk Q/I$ of dimension three and finite GK dimension. Quivers with one vertex are Artin--Schelter regular and these are classified \cite{ATV1,ASch}. Quivers with two or three vertices were determined in \cite{GR1}, and a partial classification on four vertices is given in \cite{GLNW}.

Let $A$ be a locally finite graded algebra. The \emph{graded left socle} of $A$ is $\soc_l(A) = \{ x \in A \mid J(A)x=0\}$ where $J(A)$ is the graded Jacobson radical. Similarly, $\soc_r(A) = \{ x \in A \mid xJ(A)=0\}$ is the \emph{graded right socle} of $A$. If $A$ is a derivation-quotient algebra for some $\eta$-twisted superpotential $\omega$, then $A$ is twisted graded Calabi--Yau of dimension three if and only if $h_A(t) = (p(t))\inv$ as in \eqref{eq.dim3_hilb} and $A$ has trivial graded left and right socle \cite[Lemma 8.6]{RR1}.

\section{Lifting the twisted Calabi--Yau property}
\label{sec.lift}

Let $\cD$ be a graded algebra with a regular homogeneous normal element $\omega$. Then $\omega$ induces an automorphism $\phi_\omega$ such that $\omega x = \phi_\omega(x)\omega$ for all $x \in \cD$. Let $A=\cD/\omega \cD$. In this section, we are interested in how the properties of $A$ transfer to $\cD$, in particular the twisted Calabi--Yau property.

\begin{theorem}[{\cite[Theorem 3.1]{ZSL}}]\label{thm.lift}
Keep the setup as above. If $A$ is twisted graded Calabi--Yau, then so is $\cD$. Moreover, $\mu^A = \mu^{\cD} \circ \phi_\omega$.
\end{theorem}

In case $\cD$ (and hence $A$) is connected, this follows from \cite[Proposition 2.1]{CKS}. We give an alternate proof of lifting in the setting of generalized AS regular algebras, which follows \cite{CKS} more closely. In either case, one must take care to establish regularity of the element $\omega$.

First we note the following preliminary result. The proof is given in \cite[Lemma 2.2]{CKS}. Though connectedness is stated as a hypothesis, it is not needed in the proof.

\begin{lemma}\label{lem.rees}
Let $\cD$ be an $\NN$-graded algebra with $\cD/J(\cD)$ separable, let $\omega \in \cD_m$ be a normal element of degree $m$, and let $A = \cD/\omega\cD$. For every left module $N$ and all left $\cD$-modules $M$ on which $\omega$ acts regularly, there are isomorphisms
\[  \Ext_{\cD}^{p+1}(N,M) \iso \Ext_A^p(N,{}_\tau(M/\omega M)) \]
which is functorial on both $N$ and $M$ for all $p$.
\end{lemma}

\begin{lemma}\label{lem.hilb}
Let $\cD$ be an $\NN$-graded algebra. Let $\omega \in \cD$ be a homogeneous normal element of degree $m > 0$, let $A = \cD/\omega\cD$, and let $P_\omega$ be the permutation matrix corresponding to $\phi_\omega$ (so $(P_\omega)_{ij} = \delta_{\phi_\omega(i)j}$).
\begin{enumerate}
    \item \label{hilb1} The element $\omega$ is regular if and only if the Hilbert series of $A$ and $\cD$ are related by 
    \begin{align*}
        h_{\cD}(t) = h_A(t)(I-P_\omega\inv t^m)\inv \quad\text{or}\quad
        h_{\cD}^{\tot}(t) = h_A^{\tot}(t) (1-t^m)\inv.
    \end{align*}
    \item \label{hilb2} If $\omega$ is regular, then $A$ has finite GK dimension if and only if $D$ does.
    \item \label{hilb3} The algebra $A$ is left or right noetherian if and only if $D$ is.
\end{enumerate}
\end{lemma}
\begin{proof}
\eqref{hilb1} The statement on the total Hilbert series follows almost directly from \cite[Proposition 2.1]{CKS}. In particular, 
\[ 0 \to \cD(-m) \xrightarrow[]{\cdot \omega} \cD \to A \to 0\]
is exact if and only if $\omega$ is regular. 

For the corresponding statement on the matrix-valued Hilbert series, we need only note that
\[ (e_i D e_j)\omega = e_i D\omega e_{\sigma\inv(j)} \iso e_i D e_{\sigma\inv(j)}\]
where $\sigma$ is the automorphism induced by the normal element $\omega$. It follows from the short exact sequence that $h_{\cD}(t) = h_{\cD(-m)}(t)P_\omega\inv t^m + h_A(t)$.

\eqref{hilb2} This follows from (1).

\eqref{hilb3} If $\cD$ is noetherian then so is $A$. For the converse, we apply \cite[Lemma 8.2]{ATV1}.
\end{proof}

Let $R$ be a ring with a complete set $e_1,\hdots,e_n$ of orthogonal idempotents. For example, $R=\kk Q$ with $Q$ a finite quiver. Then $R$ is a \emph{piecewise domain} (PWD) if $ab=0$ implies $a=0$ or $b=0$ for all $a \in e_i R e_j$ and $b \in e_j R e_k$ \cite{GSpwd}. By \cite[Theorem 1.4]{RR3}, if $A=\kk Q/I$ is twisted graded Calabi--Yau of dimension two and each vertex has outdegree two, then $A$ is a piecewise domain.

\begin{lemma}\label{lem.PWD2}
Let $\cD$ be an $\NN$-graded algebra with complete set of orthogonal idempotents $e_0, \dots, e_{n-1}$, let $\omega \in \cD$ be a homogeneous normal {regular} element of degree $m > 0$, and let $A = \cD/\omega\cD$. If $A$ is a piecewise domain with respect to the idempotents $e_h + \omega D$, then $D$ is a piecewise domain with respect to the $e_h$. In particular, if $A$ is a domain, then so is $D$.
\end{lemma}
\begin{proof}
Suppose that $D$ is not a piecewise domain with respect to the $e_h$. Then there exist $i,j,k \in \{ 0, \hdots, n-1 \}$ such that $a \in e_iDe_j, b \in e_j D e_k, a \neq 0, b \neq 0$, but $ab = 0$. We may assume that $a,b$ are homogeneous. We further assume that $\deg(a)+\deg(b)$ is minimal among all such counterexamples across all $i,j,k$.

If $a,b \notin \omega D$, then $a + \omega D, b + \omega D, ab + \omega D$ are all non-zero in $A$ because $A$ is a piecewise domain. So $ab \neq 0$.

Suppose $a \in \omega D$. Then there exists a homogeneous $c \in D$ such that $a = \omega c$ and $\deg(c) = \deg(a) - \deg(\omega)$. Since $ab = \omega cb = 0$ and $\omega$ is regular, we have $cb =0$.

Since $D = \oplus_{h=0}^{n-1} e_h D$, we have $e_\ell c \neq 0$ for some $\ell$. Since $b \in e_i D e_j$, $e_\ell c \in e_\ell D e_j$. By the previous direct sum decomposition of $D$, we also have that $e_\ell c b = 0$. Since $\deg(e_\ell c) + \deg(b) < \deg(a) + \deg(b)$, this contradicts the minimality condition. So $ab \neq 0$.

By normality of $\omega, \omega D = D \omega$. So if $b \in \omega D,$ we can repeat the proof with $\omega$ on the right.
\end{proof}

\begin{proposition}\label{prop.lift}
Let $\cD$ be an $\NN$-graded algebra with $\cD/J(\cD)$ separable, let $\omega \in \cD_m$ be a regular normal element of degree $m$, and let $A = \cD/(\omega)$. If $A$ is generalized AS regular of dimension $d$, then $\cD$ is generalized AS regular of dimension $d+1$.
\end{proposition}
\begin{proof}
Set ${}_{\cD} S = \cD/J(\cD)$ as a left $A$-module. Note that ${}_A S$ is the image of ${}_{\cD} S$ in the quotient. By \cite[Proposition 3.18]{RR2}, $d=\gldim(A) = \pdim(_A S)$. Since $A$ is generalized AS regular, then $\pdim({}_A S) < \infty$. The element $\omega$ is assumed regular and so applying \cite[Theorem 7.3.5(i)]{MR} gives $\pdim(_{\cD} S) = \pdim({}_A S) = d+1$. Thus, $\gldim(\cD)=d+1$. We apply Lemma \ref{lem.rees} with $N=S$ and $M=\cD$ to obtain
\[  \Ext_{\cD}^{p+1}({}_{\cD} S,\cD) \iso \Ext_A^p({}_A S,A) \]
for all $p$. The result follows.
\end{proof}

Suppose $A=\kk Q/(\omega)$ is twisted graded Calabi--Yau algebra of dimension two where $\omega=\sum_{x \in Q_1} \tau(x) x$ for some $\tau \in GL(\kk Q_1)$ and let $\mu^A$ be the Nakayama automorphism associated to $A$. Suppose $A \iso \cD/(\omega)$ where $\omega$ is a normal regular element of $D$ which induces the automorphism $\phi$ on $\kk Q$. In other words, $\cD$ is the algebra
\begin{align}\label{eq.norm_ext}
    \cD(Q,\omega,\phi) = \kk Q/(\omega a - \phi(a)\omega : a \in Q_1).
\end{align}
We call $\cD=\cD(Q,\omega,\phi)$ a \emph{normal extension of $A$}.

\begin{lemma}\label{lem.phi_restr}
Keep the above setup and suppose that $\cD=\cD(Q,\omega,\phi)$ is twisted Calabi--Yau of dimension three. Let $\phi_\omega$ be the restriction of $\phi$ to $\kk Q_0$. Then $\phi_\omega = (\mu_0^A)\inv$. 
\end{lemma}
\begin{proof}
Set $\pi=\mu_0^A$. Let $M$ be the adjacency matrix for $Q$. Then $M^T$ is the adjacency matrix for the relations on $\cD$ (see the proof of \cite[Proposition 8.2]{RR1}). Hence, if $|Q_1|=k$, then there are exactly $k$ homogeneous relations defining $D$. This implies that, for each $a \in Q_1$, $\omega a - \phi(a)\omega$ is homogeneous (otherwise the relation splits in multiple homogeneous relations). Suppose $a$ has source $r$ and target $s$. Then 
\[ \omega a = \omega (e_r a e_s) = e_{\pi\inv(r)} \omega e_r a e_s\]
and hence
\[ 
\phi(a) \omega 
    = e_{\pi\inv(r)}\phi(a) \omega e_s 
    = e_{\pi\inv(r)}\phi(a) e_{\pi\inv(s)} \omega e_s.\]
Thus, $\phi(e_r)=e_{\pi^{-1}(r)}$ and $\phi(e_s)=e_{\pi^{-1}(s)}$.
\end{proof}

\begin{proposition}\label{prop.norm_ext}
Let $A=\kk Q/(\omega)$ be a twisted graded Calabi--Yau algebra of dimension two where $\omega=\sum_{x \in Q_1} \tau(x) x$ for some $\tau \in GL(\kk Q_1)$. Let $M$ be the adjacency matrix of $Q$ and $P$ the permutation matrix corresponding to $\mu_0^A$. Let $\phi$ be a graded automorphism of $\kk Q$ and suppose that the normal extension $\cD=\cD(Q,\omega,\phi)$ as in \eqref{eq.norm_ext} is twisted graded Calabi--Yau of dimension three. Then $\cD$ has type $(M,P^2,4)$ and the following are equivalent:
\begin{enumerate}
\item \label{ext1} $A$ is noetherian.
\item \label{ext2} $A$ has finite GK dimension.
\item \label{ext3} $\cD$ is noetherian.
\item \label{ext4} $\cD$ has finite GK dimension.
\end{enumerate}
\end{proposition}
\begin{proof}
That $M$ is the adjacency matrix of $\cD$ and $\omega$ has degree four are both by construction. By Theorem \ref{thm.lift}, $\mu_0^A = \mu_0^D \circ \phi_\omega$ where $\phi_\omega$ is as defined in Lemma \ref{lem.phi_restr}. By that lemma, $\phi_\omega = (\mu_0^A)\inv$. Hence, $\mu_0^{\cD} = (\mu_0^A)^2$.

By Lemma~\ref{lem.hilb}, $\eqref{ext1} \Leftrightarrow \eqref{ext3}$ and $\eqref{ext2} \Leftrightarrow \eqref{ext4}$. Now $\eqref{ext1} \Leftrightarrow \eqref{ext2}$ by \cite[Theorem 6.6]{RR2}.
\end{proof}

Keeping the notation of Proposition \ref{prop.norm_ext}, we see that $P=P_\omega\inv$.

\section{Type A}\label{sec.typeA}

In this section, we study normal extensions of the twisted Calabi--Yau algebras $A_n(\bq)$, as defined in Definition \ref{defn.typeA}, and use tools from Section \ref{sec.lift} to prove Theorem \ref{thm.down-up}. The major results have all been shown previously in \cite{GK1}. Here we give some alternative proofs using methods from Section \ref{sec.lift} that will then be applied in Section \ref{sec.typeBC} to types B and C.

Fix $n \in \ZZ_+$ and let $\alpha,\beta \in \kk^n$. By \cite[Proposition 1.1]{GK1}, a $\kk$-basis for $\cH_n(\alpha,\beta,\gamma)$ consists of paths of the form:
\begin{align}\label{eq.Aform}
    (u_iu_{i+1}\cdots u_{i+k})(d_{i+k}u_{i+k})^j(d_{i+k}d_{i+k-1} \cdots d_{i+k-\ell}),
\end{align}
with $i \in \dblQ_0$ and $j,k,\ell \geq 0$. 

By \cite[Lemma 4.1]{GK1}, if some $\beta_i = 0$, then $\cH=\cH_n(\alpha,\beta,\gamma)$ is not right noetherian. A similar proof shows that $\cH$ is not left noetherian. The following shows that one can also appeal to the opposite ring.

\begin{proposition}\label{prop.Aop}
Let $\alpha, \beta, \gamma \in \kk^n$. Then $\cH_n(\alpha,\beta,\gamma)^\op \cong \cH_n(\alpha,\beta,\gamma)$.
\end{proposition}
\begin{proof}
Denote the elements of $\cH^\op$ with primes. Note that $u_i' = e_{i+1}'u_i'e_i'$ and $d_i' = e_i' d_i' e_{i+1}'$.

The relations of Definition~\ref{defn.qdu} in the opposite ring are
\begin{align*}
	u_i'u_{i-1}'d_{i-1}'  &= \alpha_i u_i'd_i'u_i' + \beta_i d_{i+1}'u_{i+1}'u_i'+\gamma_i u_i', \\
	u_{i-1}'d_{i-1}'d_i' &= \alpha_i d_i'u_i'd_i' + \beta_i d_i'd_{i+1}'u_{i+1}' +\gamma_i d_i'.
\end{align*}
So, the bijection $\phi:\cH \to \cH^\op$ defined by
\[ \phi(e_i) = e_i', \qquad \phi(u_i) = d_i', \qquad \phi(d_i) = u_i'\]
is an isomorphism.
\end{proof}

Fix $\bp,\bq \in (\kk^\times)^n$. Let $A=A_n(\bq)$ and let $\phi$ be the automorphism of $\kk\dblQ$ defined by 
\[
\phi(e_i) = e_{i-1}, \quad
\phi(u_i) = p_i u_i, \quad
\phi(d_i) = p_i\inv d_i.
\]
Let $\cD=\cD(\dblQ,\omega,\phi)$ with
\[ \omega = \sum_{i \in Q_0} d_iu_i - q_i u_{i+1} d_{i+1}.\]
Then the relations in $\cD$ are
\begin{align*}
\omega u_i - \phi(u_i)\omega
    &= d_{i-1}u_{i-1}u_i - (p_i+q_{i-1})u_id_iu_i + p_iq_i u_iu_{i+1}d_{i+1} \\
\omega d_i - \phi(d_i)\omega
    %&= d_iu_id_i - q_iu_{i+1}d_{i+1}d_i - p_i\inv d_id_{i-1}u_{i-1} + p_i\inv q_{i-1} d_iu_id_i \\
    &= -p_i\inv( d_id_{i-1}u_{i-1} - (p_i+q_{i-1})d_iu_id_i + p_iq_iu_{i+1}d_{i+1}d_i).
\end{align*}
These are precisely the relations given in Definition \ref{defn.qdu} 
assuming that, given $\alpha,\beta \in \kk^n$, there is a solution to the equations
\begin{align}\label{eq.alphabeta}
\alpha_i = q_{i-1} + p_i, \qquad
\beta_i = -q_i p_i.
\end{align}
for all $i \in \dblQ_0$. That is, $\omega$ is a normal element in $\cH=\cH_n(\alpha,\beta)$ and $A=\cH/(\omega)$. 

The following considers the single-parameter case on $n$ vertices in which there exists $p,q \in \kk^\times$ such that $p=p_i$ and $q=q_i$ for all $i \in \dblQ_0$. In this case, there is a solution to \eqref{eq.alphabeta}.

\begin{theorem}\label{thm.typeA}
Fix $n \in \ZZ_+$ and let $\alpha,\beta \in \kk$. Let $\cH=\cH_n(\alpha,\beta)$ be as in Definition \ref{defn.qdu}. The following are equivalent.
\begin{enumerate}
    \item $\beta \neq 0$,
    \item \label{Apwd} $\cH$ is a piecewise domain, and
    \item \label{Anoeth} $\cH$ is right (or left) noetherian.
\end{enumerate}
Moreover, if any of the above conditions hold, then $\cH$ is twisted graded Calabi--Yau of dimension three. 
\end{theorem}
\begin{proof}
Suppose $\beta \neq 0$. Set $A=\cH/(\omega)$. By 
\cite[Proposition 1.1]{GK1} (see \eqref{eq.Aform}), $h_{\cH}^{\tot}=(1-t^2)h_A^{\tot}$ and so, by Lemma \ref{lem.hilb}, $\omega$ is regular. Now \eqref{Apwd} follows from Lemma \ref{lem.PWD2} and \eqref{Anoeth} from Proposition \ref{prop.norm_ext}.

If $\beta = 0$, then $(d_{i-1}u_{i-1}-\alpha_i u_id_i)u_i = 0$, so \eqref{Apwd} fails, and \eqref{Anoeth} fails by \cite[Lemma 4.1]{GK1}.

Finally, given the equivalent conditions, $\cH$ is twisted graded Calabi--Yau of dimension three by Proposition \ref{prop.lift}.
\end{proof}

Theorem \ref{thm.typeA} can be generalized to the multiparameter case with $\alpha,\beta \in \kk^n$ where (1) is replaced by $\beta_i \neq 0$ for all $i\in\dblQ_0$. Unfortunately, for some choices of $\alpha,\beta$, there is no solution to \eqref{eq.alphabeta}, as the following example shows.

\begin{example}
Let $n > 1$. Set $\beta_i = 1$ for $i \in \dblQ_0$, $\alpha_0 \neq 0$, and $\alpha_i = 0$ for $i = 1, \dots, n-1$. Then $\alpha_i = q_{i-1} + p_i, \beta_i = -q_i p_i$ cannot be solved simultaneously for all $q_i, p_i$.
\end{example}
\begin{proof} 
Suppose that there is a solution. Then for all $i \in \dblQ_0$, $\beta_i = 1$ implies $q_i = -\frac{1}{p_i}$. Then for $i \in \dblQ_0$, 
\[ \alpha_i = -\frac{1}{p_{i-1}}+ p_i = 0 \implies p_i = \frac{1}{p_{i-1}}.\]
Thus for all $j,k \in \dblQ_0$, $p_j = p_k$ when $j,k$ are both odd or both even.

If $n$ is odd, then $p_0 = p_{n-1}$. But also $p_0 = p_n$ by the convention that the indices are taken mod $n$. Thus $p_0 = p_i$ for all $i = 0, \dots, n_1$. This implies that the $\alpha_i$ are constant, a contradiction.

If $n$ is even, then $p_0 = p_{n-2}$ and $p_1 = p_{n-1}$. We have
\[
\alpha_{n-1} = 0 = -\frac{1}{p_0} + p_1,
\quad 0 \neq \alpha_0 = -\frac{1}{p_1} + p_0,
\]
which is also a contradiction.
\end{proof}

We now establish certain conditions that guarantee a solution for the $p_i$ and $q_i$. In these cases, we obtain an analogous result to Theorem \ref{thm.typeA}.

In \cite[Lemma~5.1]{GK1},  basic isomorphisms $\phi:\cH_n(\alpha,\beta) \to \cH_n(\alpha', \beta')$ are described, with $\alpha, \beta, \alpha', \beta' \in \kk^n$. Further, if $n \geq 3$, any isomorphism is a composition of such $\phi$ \cite[Theorem~5.2]{GK1}. We now show that one can solve \eqref{eq.alphabeta} given a pair $\alpha, \beta$ if and only if the equations for a pair $\alpha', \beta'$, given by one of the isomorphisms of \cite[Lemma~5.1]{GK1}. We need not consider the actual isomorphisms, just their effect on $\alpha, \beta$.

\begin{lemma}\label{lem:solvePQ}
Let $\alpha, \beta, \alpha', \beta' \in \kk^n$. Suppose there exist $p,q \in \kk^n$ that solve \eqref{eq.alphabeta}. Then there exist $p', q' \in \kk^n$ solving $\alpha'_i = q'_{i-1} + p'_i, \beta'_i = -p'_i q'_i$ in the following cases.
\begin{enumerate}
\item \label{solve1} If there exists $\lambda \in (\kk^\times)^n$ such that for all $i \in \dblQ_0$, \[
\alpha'_i = \lambda_i \lambda_{i-1}^{-1}\alpha_i,\quad  \beta'_i = \lambda_{i+1} \lambda_{i-1}^{-1}\beta_i,
\]
then
\[
p_i' = \lambda_i \lambda_{i-1}^{-1} p_i, \quad q_i' = \lambda_{i+1}\lambda_i^{-1} q_i
\]
is a solution.
\item \label{solve2} If for all $i \in \dblQ_0$,
\[
\alpha_i' = \alpha_{i-1}, \quad \beta_i' = \beta_{i-1},
\]
then
\[
p_i' = p_{i-1}, \quad q_i' = q_{i-1}
\]
is a solution.
\item \label{solve3} If for all $i \in \dblQ_0$, $\beta_i \neq 0$ and
\[
\alpha_i' = -\beta_{n-i-1}^{-1}\alpha_{n-i-1}, \quad \beta_i' = \beta_{n-i-1}^{-1},
\]
then
\[
p_i' = -q_{n-i-1}^{-1}q_{n-i-2}p_{n-i-1}^{-1}, \quad q_i' = -q_{n-i-2}^{-1}
\]
is a solution.
\end{enumerate}
\end{lemma}
\begin{proof}
In case \eqref{solve1}, note that $q_{i-1}' = \lambda_{i}\lambda_{i-1}^{-1} q_{i-1}$. Then
\[
\alpha_i' = \lambda_i \lambda_{i-1}^{-1} \alpha_i = \lambda_i \lambda_{i-1}^{-1} (q_{i-1} + p_i) = q_{i-1}' + p_i'
\]
and
\[
\beta_i' = \lambda_{i+1}\lambda_{i-1}^{-1} \beta_i = \lambda_{i+1}\lambda_{i-1}^{-1} (-p_i q_i) = - (\lambda_i \lambda_{i-1}^{-1} p_i)(\lambda_{i+1}\lambda_i^{-1}q_i) = - p_i' q_i'.
\]

Case \eqref{solve2} is clear.

For case \eqref{solve3}, note that $q_{i-1}' = -q_{n-i-1}^{-1}$. Then
\[
\alpha_i' = -\beta_{n-i-1}^{-1}\alpha_{n-i-1} = -p_{n-i-1}^{-1}q_{n-i-1}^{-1}(q_{n-i-2} + p_{n-i-1}) = p_i' + q_{i-1}'
\]
and
\[
\beta_i' = \beta_{n-i-1}^{-1} = -p_{n-i-1}^{-1} q_{n-i-1}^{-1} = - (-p_{n-i-1}^{-1} q_{n-i-1}^{-1} q_{n-i-2})(- q_{n-i-2}^{-1}) = -p_i' q_i'.
\]
So, in all cases, we have a solution.
\end{proof}

\begin{proposition}\label{prop.noSoln}
Fix $n \in \ZZ_+$ with $n \geq 3$ and let $\alpha, \beta, \alpha', \beta' \in \kk^n$. Suppose $\cH_n(\alpha,\beta) \cong \cH_n(\alpha', \beta')$. Then \eqref{eq.alphabeta} has a solution if and only if the analogous prime equations have a solution.
\end{proposition}
\begin{proof}
By \cite[Theorem 5.2]{GK1}, if $\phi:\cH_n(\alpha,\beta) \to \cH_n(\alpha', \beta')$ is an isomorphism, then $\alpha_i', \beta_i'$ have formulas that are compositions of the formulas in the hypotheses Lemma~\ref{lem:solvePQ}. Thus a solution for the $\alpha, \beta$ implies a solution for the $\alpha', \beta'$, and conversely by considering $\phi^{-1}$.
\end{proof}

Since Proposition \ref{prop.noSoln} shows that sometimes we cannot find appropriate $p_i, q_i$, we examine when we still have normal regular elements of degree $2$.

\begin{proposition}
Fix $n \in \ZZ_+$ and let $\alpha, \beta\in \kk^n$.
Let $\cH = \cH_n(\alpha, \beta)$ with $\beta_i \neq 0$ for all $i\in\dblQ_0$. Let $u = \sum_{i=0}^{n-1} u_i, d = \sum_{i=0}^{n-1} d_i$. Then $u^2$ and $d^2$ are normal  if and only if $\alpha_i = 0$ for all $i$. In this case of all $\alpha_i = 0$, $u^2$ and $d^2$ are also regular.
\end{proposition}
\begin{proof}
We consider $u^2$, with the proof for $d^2$ being similar.

First, $u_i u^2 = u_i u_{i+1} u_{i+2} = u^2 u_{i+2}$ and $u^2 u_i = u_{i-2} u^2$. So, normality holds when multiplying by $u_i$.

Second, $d_i u^2 = d_i u_i u_{i+1}$, which has source $i+1$ and target $i+2$. The only polynomial of degree $3$ that begins with $u^2$ is $a u^2 d_{i+2} = a u_{i+1}u_{i+2}d_{i+2}$ for some $a \in \kk^\times$. Thus, via
\begin{align*}\label{eq:relationsA}
	d_{i-1}u_{i-1}u_i &= \alpha_i u_id_iu_i + \beta_i u_iu_{i+1}d_{i+1} +\gamma_i u_i \\
	d_id_{i-1}u_{i-1} &= \alpha_i d_iu_id_i + \beta_i u_{i+1}d_{i+1}d_i +\gamma_i d_i,
\end{align*}
we have
\[
d_i u_i u_{i+1} =a u_{i+1}u_{i+2}d_{i+2} =  \alpha_i u_id_iu_i + \beta_i u_iu_{i+1}d_{i+1}.
\]
Thus, $d_i u^2 = u^2 (\beta_i d_{i+2})$ if and only if $\alpha_i = 0$ and $a = \beta_i$. Since all $\beta_i \neq 0$, we also have $u^2 d_i = (\beta_{i-2}^{-1} d_{i-2} u^2)$. Therefore, $u^2$ is normal if and only if all $\alpha_i = 0$.

For regularity, let $f \in \cH$ with $f \neq 0$. Writing $f$ in the basis \eqref{eq.Aform}, we see that $u^2 f$ is also of that form. Since $f$ has nonzero basis coefficients, so does $u^2 f$.

Now relations in $\cH = \cH(0, \beta)$ are defined by \eqref{eq:relationsA} with $\alpha_i = \gamma_i = 0$ and $\beta_i \neq 0$. Multiplying by $\beta_i^{-1}$, we have relations that put $d$'s first, so we also have a basis with the ordering of \eqref{eq.Aform} reversed. Thus, $fu^2 \neq 0$ and so $u^2$ is regular.
\end{proof}

The restriction that $\alpha_i = \gamma_i = 0$ and $\beta_i \neq 0$ for all $i\in\dblQ_0$ is broad enough that sometimes we can solve for $p_i, q_i$ and sometimes not.

\begin{proposition}\label{prop.alphabeta}
Let $n > 0$. Suppose $\alpha_i = 0$ and $\beta_i \neq 0$ for all $i\in\dblQ_0$. There exists a solution to \eqref{eq.alphabeta} if and only if $n$ is odd or if $n$ is even and
\[
\beta_0 \beta_2 \dots \beta_{n-2}
= \beta_1 \beta_3 \dots \beta_{n-1}.
\]
\end{proposition}
\begin{proof}
Suppose there is a solution. Then for all $i\in\dblQ_0$, 
\[
q_i = -\beta_i p_i^{-1}, \quad p_i = \beta_{i-1}^{-1} p_{i-1}^{-1}.
\]
Then induction gives
\[
p_i = \beta_{i-1}\beta_{i-3} \dots \beta_1 (\beta_{i-2} \dots \beta_0)^{-1} p_0^{(-1)^i}.
\]
Taking $i = n$ and noting $p_0 = p_n$, we have
\begin{equation}\label{eq.p0eqn}
p_0 = \beta_{n-1}\beta_{n-3} \dots \beta_1 (\beta_{n-2} \dots \beta_0)^{-1} p_0^{(-1)^n}.
\end{equation}
Conversely, if \eqref{eq.p0eqn} has a solution, then we can back substitute to get a solution for all $p_i, q_i$.

If $n$ is odd, then \eqref{eq.p0eqn} has the solution $p_0 = \pm \sqrt{\beta_{n-1}\beta_{n-3} \dots \beta_1 (\beta_{n-2} \dots \beta_0)^{-1}}$. If $n$ is even and 
$\beta_0 \beta_2 \dots \beta_{n-2}
= \beta_1 \beta_3 \dots \beta_{n-1}$, then any $p_0 \neq 0$ works. On the other hand, if $\beta_0 \beta_2 \dots \beta_{n-2}
\neq \beta_1 \beta_3 \dots \beta_{n-1}$, there is no solution.
\end{proof}

\section{Types B and C}
\label{sec.typeBC}

In this section, we apply the techniques above, used in type A, to type B and C. In particular, we prove Theorem \ref{thm.down-up} for type B in Theorem \ref{thm.typeB} and for type C in Theorem \ref{thm.typeC}. 

\subsection{Type B}
First we consider type B as described in Definition~\ref{defn.qDU_B}.

\begin{lemma}\label{lem.Bbasis}
Fix $n \geq 2$ and let $\alpha,\beta \in \kk$. 
For any $i\in\tldQ_0$, set $x_i = b_i a_i$. A $\kk$-basis for $\cB=\cB_n(\alpha,\beta)$ consists of paths of the form:
\begin{equation}\label{eq.BbasisWithXi}
a_ia_{i+1}a_{i+2}\cdots a_{i+(k-1)}
x_{i+k} x_{i+k+1} \cdots x_{i+k+(\ell-1)}b_{i+k+\ell}^m
\end{equation}
with $i\in\tldQ_0$ and $k,\ell,m\geq 0$. Consequently, $h_{\cB}^{\tot}(t) = (1-t)^{-2}(1-t^2)\inv$.
\end{lemma}
\begin{proof}
We order the paths by $b_0 > b_1 > \cdots > b_{n-1} > a_0 > a_1 > \cdots > a_{n-1}$ so that the leading terms of the defining ideal are $b_ia_ia_{i+1}$ and $b_i^2a_i$ for $i \in \tldQ_0$. Thus, the only overlap ambiguities are of the form $b_i^2a_ia_{i+1}$. We verify below that these ambiguities resolve:
\begin{align*}
(b_i^2a_i)a_{i+1} - b_i(b_ia_ia_{i+1})
        &= (\alpha b_ia_ib_{i+1} + \beta a_ib_{i+1}^2)a_{i+1} - b_i(\alpha a_ib_{i+1}a_{i+1} + \beta a_ia_{i+1}b_{i+2}) \\
        %&= \alpha b_ia_ib_{i+1}a_{i+1} + \beta a_i(b_{i+1}^2a_{i+1})
            %- \alpha b_ia_ib_{i+1}a_{i+1} - \beta (b_ia_ia_{i+1})b_{i+2} \\
        &= \beta (a_i(b_{i+1}^2a_{i+1}) - (b_ia_ia_{i+1})b_{i+2}) \\
        &= \beta \left( a_i ( \alpha b_{i+1}a_{i+1}b_{i+2} + \beta a_{i+1}b_{i+2}^2 )
        		- ( \alpha a_ib_{i+1}a_{i+1} + \beta a_ia_{i+1}b_{i+2} ) b_{i+2}\right) = 0.
\end{align*}
Hence, a $\kk$-basis for $\cB$ consists of those paths that avoid the leading terms. It is straightforward to see that these are the paths in \eqref{eq.BbasisWithXi}.

There is a bijection between paths with source $e_i$ and monomoials $x^k y^\ell z^m$ in $\kk[x,y,z]$ where $x,z$ are given degree 1 and $y$ has degree 2. The statement on the total Hilbert series is now clear.
\end{proof}

\begin{remark}\label{rmk.Bparams}
Fix $n \geq 2$ and let $\alpha,\beta \in \kk^n$. Let $\cB_n(\alpha,\beta)$ be the quotient of $\kk\tldQ$ by the relations
\begin{align*}
	b_ia_ia_{i+1} &= \alpha_i a_ib_{i+1}a_{i+1} + \beta_i a_ia_{i+1}b_{i+2} \\
    b_i^2a_i &= \alpha_i b_ia_ib_{i+1} + \beta_i a_ib_{i+1}^2,
\end{align*}
for $i \in \tldQ_0$. An argument as in Lemma \ref{lem.Bbasis} shows that the ambiguities resolve if and only if $\alpha_i=\alpha_{i+1}$ and $\beta_i=\beta_{i+1}$ for all $i \in Q_0$. Hence, in contrast to type A, for type B there is only a $2$-parameter family of algebras for each $n$.
\end{remark}

Generically, the class of type B algebras is preserved under taking opposite rings. However, when $\beta = 0$, this fails.

\begin{proposition}
\label{prop.Bop}
Let $\alpha, \beta \in \kk$ with $\beta \neq 0$. Then $\cB(\alpha, \beta)^\op \cong \cB(-\beta^{-1}\alpha,  \beta^{-1})$.
\end{proposition}
\begin{proof}
We denote the elements of $\cB^\op$ with primes. Note that $a_{-i}' = e_{-i+1}' a_{-i}' e_{-i}'$ and $b_{-i}' = e_{-i}' b_{-i}' e_{-i}'$.
The relations of Definition~\ref{defn.qDU_B} in the opposite ring are
\begin{align*}
	a_{i+1}'a_i'b_i' &= \alpha a_{i+1}'b_{i+1}'a_i' + \beta b_{i+2}'a_{i+1}'a_i' \\
    a_i'{b'}_i^2 &= \alpha b_{i+1}'a_i'b_i' + \beta {b'}_{i+1}^2a_i',
\end{align*}
Multiplying by $\beta^{-1}$ and rearranging, we have
\begin{align*}
	b_{i+2}'a_{i+1}'a_i'  &= -\beta^{-1}\alpha a_{i+1}'b_{i+1}'a_i' + \beta^{-1} a_{i+1}'a_i'b_i' \\
     {b'}_{i+1}^2a_i' &= -\beta^{-1}\alpha b_{i+1}'a_i'b_i'  + \beta^{-1} a_i'{b'}_i^2,
\end{align*}
Replacing $i$ with $-i-1$ in the first equation and replacing $i$ with $-i$ in the second gives
\begin{align*}
	b_{-i+1}'a_{-i}'a_{-i-1}'  &= -\beta^{-1}\alpha a_{-i}'b_{-i}'a_{-i-1}' + \beta^{-1} a_{-i}'a_{-i-1}'b_{-i-1}' \\
     {b'}_{-i+1}^2a_{-i}' &= -\beta^{-1}\alpha b_{-i+1}'a_{-i}'b_{-i}'  + \beta^{-1} a_{-i}'{b'}_{-i}^2,
\end{align*}
Then the bijection $\phi:\cB(-\beta^{-1}\alpha, \beta^{-1}) \to \cB(\alpha,\beta)^\op$ given by
\[
\phi(e_i) = e'_{-i+1}, \qquad \phi(a_i) = a'_{-i}, \qquad \phi(b_i) = b'_{-i+1}
\]
is an isomorphism.
\end{proof}

When $\beta = 0$, the opposite algebra of a Type B algebra is not Type B. 

\begin{proposition}
Let $\alpha \in \kk$. For all $\alpha', \beta' \in \kk$, $\cB(\alpha, 0) \ncong \cB(\alpha', \beta')^\op$ as $\kk$-algebras.
\end{proposition}
\begin{proof}
Suppose for contradiction that there exists $\alpha', \beta' \in \kk$ and a $\kk$-algebra isomorphism $\phi:\cB(\alpha', \beta')^\op \to \cB(\alpha, 0)$. If $\beta' \neq 0$, then $\cB(\alpha', \beta')^\op$ is noetherian, while $\cB(\alpha, 0)$ is not. See Theorem~\ref{thm.typeB} below. Thus $\beta' = 0$.

Let $e_i', a_i', b_i'$ be the generators of $\cB(\alpha', \beta')^\op$ (using the opposite relations), and $e_i, a_i, b_i$ be the generators of $\cB(\alpha, 0)$ for $i=0, \dots, n-1$.
By \cite[Theorem~8, Lemma~4]{Gpath}, we may assume that $\phi$ is graded and $\phi(e_i') = e_{-i+1}$ for all $i$. Up to nonzero constant multiples, the only degree 1 elements with the same source and target are the $b_i, b_i'$. Thus,
\[
\phi(b_i') = \phi(e_i')\phi(b_i')\phi(e_i') = e_{-i+1} \lambda_i b_{-i+1} e_{-i+1} = \lambda_i b_{-i+1}
\]
with $\lambda_i \in \kk^\times$. 

Similarly, the only degree 1 elements with source and target differing by 1 index are the $a_i, a_i'$, up to nonzero constant multiples. So, there are $\delta_i \in \kk^\times$ such that
\[
\phi(a_i') = \phi(e_{i+1}')\phi(a_i')\phi(e_i') = e_{-i} \phi(a_i') e_{-i+1} =  \delta_i a_{-i}.
\]

The relations of Definition~\ref{defn.qDU_B} in the opposite ring $\cB(\alpha',0)^\op$ are
\begin{align*}
	a_{i+1}'a_i'b_i' &= \alpha' a_{i+1}'b_{i+1}'a_i'  \\
    a_i'{b'}_i^2 &= \alpha' b_{i+1}'a_i'b_i',
\end{align*}
which mapped to $\cB(\alpha, 0)$ are
\begin{align*}
(\delta_{i+1}\delta_i \lambda_i)	a_{-i-1}a_{-i}b_{-i+1} &= (\delta_{i+1}\lambda_{i+1}\delta_i) \alpha' a_{-i-1}b_{-i}a_{-i}  \\
    (\delta_i \lambda_i^2) a_{-i}{b}_{-i+1}^2 &= (\lambda_{i+1}\delta_i \lambda_i) \alpha' b_{-i}a_{-i}b_{-i+1}.
\end{align*}
Either equation contradicts the linear independence of the basis elements in Lemma~\ref{lem.Bbasis}, so $\phi$ cannot exist.
\end{proof}

\begin{lemma}\label{lem.Bpot}
Fix $n \geq 2$ and let $\alpha,\beta \in \kk$ with $\beta \neq 0$. Define $\eta \in \Aut_{\gr}(\kk\tldQ)$ by 
\[ 
\eta(e_i) = e_{i-2}, \quad
\eta(a_i)=\beta a_{i-2}, \quad
\eta(b_i)=\beta\inv b_{i-2}.
\]
Let $m$ be the number of orbits of $\eta$ on $\tldQ_0$.
Then $\cB_n(\alpha,\beta)$ is the derivation-quotient on $\kk\tldQ$ with $\eta$-twisted superpotential
\[ \Omega = \sum_{i=0}^{m-1} [b_i^2a_ia_{i+1}] - \alpha[b_ia_ib_{i+1}a_{i+1}].\]
\end{lemma}
\begin{proof}
Observe that
\[ 
[b_0^2a_0a_1]
    = b_0^2a_0a_1 - \beta a_{n-1}b_0^2a_0 + \beta^2 a_{n-2}a_{n-1}b_0^2 - \beta b_{n-2}a_{n-2}a_{n-1}b_0 + b_{n-2}^2a_{n-2}a_{n-1} + \cdots
\]
If $n$ is odd, then it is clear that
\[
[b_0^2a_0a_1] = \sum_{i=0}^{n-1} b_i^2a_ia_{i+1} - \beta a_{i-1}b_i^2a_i + \beta^2 a_{i-2}a_{i-1}b_i^2 - \beta b_{i-2}a_{i-2}a_{i-1}b_i.
\]
Similarly,
\[
[b_0a_0b_1a_1] = \sum_{i=0}^{n-1} b_ia_ib_{i+1}a_{i+1}
    - \beta a_{i-1}b_ia_ib_{i+1}.
\]
The result for $n$ odd is now clear. The result for $n$ even is similar but there are two orbits.
\end{proof}

\begin{lemma}\label{lem.Bnnoeth}
Fix $n \geq 2$ and let $\alpha \in \kk$. Then $\cB=\cB_n(\alpha,0)$ is not a piecewise domain and not noetherian.
\end{lemma}
\begin{proof}
Since $\beta=0$, then the relations in Definition \ref{defn.qDU_B} are
\begin{align*}
0 &= b_ia_ia_{i+1}-\alpha a_ib_{i+1}a_{i+1} = (b_ia_i-\alpha a_ib_{i+1})a_{i+1} \\
0 &= b_i^2 a_i - \alpha b_ia_ib_{i+1} = b_i(b_i a_i - \alpha a_ib_{i+1})
\end{align*}
Thus, $\cB$ is not a piecewise domain. To show that $\cB$ is not right noetherian, we follow the proofs of \cite[Lemma 4.3]{KMP} and \cite[Lemma~4.1]{GK1}.

Since $\cB = \oplus_{j=0}^{n-1} e_j \cB$, it suffices to show that $e_0 \cB$ is not noetherian as a right $\cB$-module. By Lemma~\ref{lem.Bbasis}, $e_0 \cB$ has a $\kk$-basis of the form \eqref{eq.BbasisWithXi} with $i=0$.

Since $\beta = 0$, the relations in Definition~\ref{defn.qDU_B} yield
\begin{equation}\label{eq.ABkill}
(\alpha a_i b_{i+1} - x_{i})a_{i+1} 
    = b_i(\alpha a_i b_{i+1}- x_{i}) = 0.
\end{equation}
Note that if $a_k \neq a_{i+1}$, then $(\alpha a_i b_{i+1} - x_{i})a_k = 0$ because of the source and target mismatch. Then
\begin{equation}\label{eq.UABkill}
(\alpha a_0 b_1 - x_0)a_i = 0 \text{ for all }i \in \ZZ.
\end{equation}

Set $U= a_0 a_1 \cdots a_{n-1}$. For each $s \geq 1$, define
\[ I_s = \sum_{m=1}^s U^m(\alpha a_0 b_1 - x_0)\cB.\]
By \eqref{eq.UABkill}, 
\[ I_s = \sum_{m=1}^s \sum_{j,k \geq 0} \kk U^m (\alpha a_0 b_1 - x_0) x_1 x_2 \cdots x_j b_{j+1}^k,\]
with the convention that when $j=0$, $x_1 x_2 \cdots x_j = 1$.

By \eqref{eq.ABkill}, $b_i x_{i} = b_i (\alpha a_i b_{i+1}) = x_i(\alpha b_{i+1})$ and hence 
\[
b_1 (x_1 x_2 \cdots x_{j}) = \alpha^j (x_1 x_2 \cdots x_{j}) b_{j+1}
\] for any $j \geq 0$.
Thus $I_s$ is contained in the $\kk$-span of the basis monomials
\[
U^m x_0 x_1 x_2 \cdots x_{j} b_{j+1}^k, \qquad 
U^m a_0 x_1 x_2 \cdots x_{j} b_{j+1}^{k+1} 
\]
with $k, j \geq 0$ and $1 \leq m \leq s$. Thus no polynomial in $I_s$ has $U^{s+1} x_{0}$ in its support, while $I_{s+1}$ does contain the polynomial $U^{s+1}(\alpha a_0 b_1 - x_{0})$.

Thus $I_s \subsetneq I_{s+1}$ for all $s \geq 1$ and so $e_0 \cB$ is not noetherian. So $\cB$ is not right noetherian. 

A similar proof shows that $\cB e_1$ is not a noetherian left $\cB$-module. For each $s \geq 1$, set
\[ J_s = \sum_{m=1}^s \cB(\alpha a_0 b_1 -x_0)b_1^m.\]
By \eqref{eq.ABkill} and source-target mismatch, we have $b_i(\alpha a_0 b_1 - x_0) = 0$ for all $i$. So
\[ J_s = \sum_{m=1}^s \sum_{j,k \geq 0} \kk a_{-j-k} \cdots a_{-j-1}x_{-j} \cdots x_{-1}(\alpha a_0 b_1 - x_0)b_1^m,\]
with the conventions that when $j=0$, $x_{-j} \cdots x_{-1} = 1$ and when $k=0$, $a_{-j-k} \cdots a_{-j-1} = 1$.

By \eqref{eq.ABkill}, $x_i a_{i+1} = (\alpha a_i) x_{i+1}$. Thus
\[
(x_{-j} \cdots x_{-1})a_0 = \alpha^j a_{-j} (x_{-j+1} \cdots x_{0}).
\]
So $J_s$ is contained in the $\kk$-span of
\[
a_{-j-k} \cdots a_{-j-1}x_{-j} \cdots x_{-1} x_0 b_1^m, \qquad a_{-j-k} \cdots a_{-j}x_{-j+1} \cdots x_{-1} x_0 b_1^{m+1}
\]
with $k, j \geq 0$ and $1 \leq m \leq s$.
Note that when $k=j=0$, these simplify to $x_0 b_1^m$ and $a_0 b_1^{m+1}$. So $x_0 b_1^{s+1}$ is in the support of a polynomial in $J_{s+1}$, but not in $J_s$. Thus $\cB e_1$ is not noetherian and $\cB$ is not left noetherian.
\end{proof}

Fix $p,q \in \kk^\times$. Let $B=B_n(q)$ and let $\phi$ be the automorphism of $\kk\tldQ$ defined by 
\[
\phi(e_i) = e_{i-1}, \quad
\phi(a_i) = p a_{i-1}, \quad
\phi(b_i) = p\inv b_{i-1}.
\]
Let $\cD=\cD(\tldQ,\omega,\phi)$ with
\[ \omega = \sum_{i \in Q_0} b_ia_i - q a_i b_{i+1}.\]
Then the relations in $\cD$ are
\begin{align*}
\omega a_i - \phi(a_i)\omega
	&= b_{i-1}a_{i-1}a_i - (q+p) a_{i-1}b_ia_i + pq a_{i-1} a_i b_{i+1} \\
\omega b_i - \phi(b_i)\omega
	%&= b_{i-1}a_{i-1}b_i - q a_{i-1}b_i^2 - p\inv b_{i-1}^2a_{i-1}  + p\inv q b_{i-1}a_{i-1}b_i
	&= -p\inv(b_{i-1}^2a_{i-1} - (q+p)b_{i-1}a_{i-1}b_i + pq a_{i-1}b_i^2).
\end{align*}
These are precisely the relations given in Definition \ref{defn.qDU_B} with $\alpha=q+p$ and $\beta=-qp$. That is, $\omega$ is a normal element in $\cB$ and $B=\cB/(\omega)$. Given $\alpha,\beta \in \kk$, the $p,q$ that give $\alpha,\beta$ in the above construction are precisely the roots of the polynomial $t^2-\alpha t - \beta$.

\begin{theorem}\label{thm.typeB}
Fix $n \geq 2$ and let $\alpha,\beta \in \kk$. Let $\cB=\cB_n(\alpha,\beta)$. The following are equivalent:
\begin{enumerate}
\item $\beta \neq 0$,
\item \label{Bpwd} $\cB$ is a piecewise domain, and
\item \label{Bnoeth} $\cB$ is right (or left) noetherian.
\end{enumerate}
Moreover, if any of the above conditions hold, then $\cB$ is twisted graded Calabi--Yau of dimension three. 
\end{theorem}
\begin{proof}
Suppose $\beta \neq 0$. Set $B=\cB/(\omega)$. By Lemma \ref{lem.Bbasis}, $h_{\cB}^{\tot}=(1-t^2)h_B^{\tot}$ and so, by Lemma \ref{lem.hilb}, $\omega$ is regular. Now \eqref{Bpwd} follows from Lemma \ref{lem.PWD2} and \eqref{Bnoeth} from Proposition \ref{prop.norm_ext}. The converse now follows from Lemma \ref{lem.Bnnoeth}. Finally, given the equivalent conditions, $\cB$ is twisted graded Calabi--Yau of dimension three by Proposition \ref{prop.lift}.
\end{proof}

\begin{corollary}
Fix $\alpha \in \kk$ and $\beta \in \kk^\times$. Let $\cB=\cB_n(\alpha,\beta)$. The Nakayama automorphism $\mu^{\cB}$ is given by 
\[ \mu^{\cB}(e_i)=e_{i+2}, \quad \mu^{\cB}(a_i)=\beta\inv a_{i+2}, \quad \mu^{\cB}(b_i)=\beta b_{i+2}.\]
The matrix-valued Hilbert series of $\cB$ is $h_{\cB}(t)=(I-Mt+P^Bt^2)\inv(I-P^Bt^2)\inv$ where $P^B$ is defined as in \eqref{eq.PtypeB}.
\end{corollary}

\subsection{Type C}
We now consider quiver down-up algebras of type C, as described in Definition~\ref{defn.qDU_C}.

\begin{lemma}\label{lem.Cbasis}
Fix $n \in \ZZ_+$ and let $\alpha,\beta \in \kk$. 
For any $i\in\hatQ_0$, set $x_i = b_i a_{i+1}$.
A $\kk$-basis for $\cC=\cC_n(\alpha,\beta)$ consists of paths of the form:
\begin{equation}\label{eq.CbasisWithXi}
a_ia_{i+1}a_{i+2}\cdots a_{i+(k-1)}
x_{i+k}x_{i+k+2} \cdots x_{i+k+(\ell-2)}
b_{i+k+\ell}b_{i+k+\ell+1} \cdots b_{i+k+\ell+(m-1)}
\end{equation}
with $i\in\hatQ_0$ and $\ell,m,k\geq0$. Consequently, $h_{\cC}^{\tot}(t) = (1-t)^{-2}(1-t^2)\inv$.
\end{lemma}
\begin{proof}
We order the paths by $b_0 > b_1 > \cdots > b_{n-1} > a_0 > a_1 > \cdots > a_{n-1}$ so that the leading terms of the defining ideal are $b_ia_{i+1}a_{i+2}$ and $b_ib_{i+1}a_{i+2}$ for $i \in \hatQ_0$. Thus, the only overlap ambiguities are of the form $b_ib_{i+1}a_{i+2}a_{i+3}$. We verify below that these ambiguities resolve:
\begin{align*}
(b_ib_{i+1}a_{i+2})a_{i+3} &- b_i(b_{i+1}a_{i+2}a_{i+3}) \\
	&= (\alpha b_ia_{i+1}b_{i+2} + \beta a_ib_{i+1}b_{i+2})a_{i+3}
		- b_i(\alpha a_{i+1}b_{i+2}a_{i+3} + \beta a_{i+1}a_{i+2}b_{i+3}) \\
	&= \beta\left( a_i(\alpha b_{i+1}a_{i+2}b_{i+3} + \beta a_{i+1}b_{i+2}b_{i+3}) 
		- (\alpha a_ib_{i+1}a_{i+2} + \beta a_ia_{i+1}b_{i+2})b_{i+3}\right) = 0.
\end{align*}
Hence, a $\kk$-basis for $\cC_n(\alpha,\beta)$ consists of those paths that avoid the leading terms. It is straightforward to see that these are paths in 
\eqref{eq.CbasisWithXi}.

There is a bijection between paths with source $e_i$ and monomoials $x^k y^\ell z^m$ in $\kk[x,y,z]$ where $x,z$ are given degree 1 and $y$ has degree 2. The statement on the total Hilbert series is now clear.
\end{proof}

As in Remark \ref{rmk.Bparams}, it is not possible to choose parameters more generally in type C. Like in the type A case, the class of type C algebras is preserved under taking opposite rings.

\begin{proposition}\label{prop.Cop}
Let $\alpha, \beta \in \kk$. Then $\cC(\alpha, \beta)^\op \cong \cC(\alpha, \beta)$.
\end{proposition}
\begin{proof}
Denote the elements of $\cC^\op$ with primes. Note that $a_{-i}' = e_{-i+1}' a_{-i}' e_{-i}'$ and $b_{-i}' = e_{-i+1}' b_{-i}' e_{-i}'$.

The relations of Definition~\ref{defn.qDU_C} in the opposite ring, after replacing $i$ with $-i-2$, become
\begin{align*}
    a_{-i}'a_{-i-1}'b_{-i-2}' &= \alpha a_{-i}'b_{-i-1}a_{-i-2}' + \beta b_{-i}'a_{-i-1}'a_{-i-2}' \\
	a_{-i}'b_{-i-1}'b_{-i-2}' &= \alpha b_{-i}'a_{-i-1}'b_{-i-2}' + \beta b_{-i}'b_{-i-1}'a_{-i-2}'.
\end{align*}
So the bijection $\phi:\cC \to \cC^\op$ defined by
\[ \phi(e_i) = e_{-i+1}', \qquad \phi(a_i) = b_{-i}', \qquad \phi(b_i) = a_{-i}'\]
is an isomorphism.
\end{proof}

The proof of the following is similar to Lemma \ref{lem.Bpot}.

\begin{lemma}\label{lem.Cpot}
Fix $n \in \ZZ_+$ and let $\alpha,\beta \in \kk$ with $\beta \neq 0$. Define $\eta \in \Aut_{\gr}(\kk\hatQ)$ by 
\[ 
\eta(e_i) = e_{i-4}, \quad
\eta(a_i)=\beta a_{i-4}, \quad
\eta(b_i)=\beta\inv b_{i-4}.
\]
Let $m$ be the number of orbits of $\eta$ on $\hatQ_0$.
Then $\cC_n(\alpha,\beta)$ is the derivation-quotient on $\kk\hatQ$ with $\eta$-twisted superpotential
\[
\Omega = \sum_{i=0}^{m-1} [b_ib_{i+1}a_{i+2}a_{i+3}] - \alpha[b_ia_{i+1}b_{i+2}a_{i+3}].
\]
\end{lemma}

\begin{lemma}\label{lem.Cnnoeth}
Let $\alpha \in \kk$. Then $\cC=\cC_n(\alpha,0)$ is not a piecewise domain and not noetherian.
\end{lemma}
\begin{proof}
Since $\beta=0$, then the relations in Definition \ref{defn.qDU_C} are
\begin{align*}
0 &= (b_ia_{i+1} - \alpha a_ib_{i+1})a_{i+2} \\
0 &= b_i(b_{i+1}a_{i+2} - \alpha a_{i+1}b_{i+2}).
\end{align*}
Thus, $\cC$ is not a piecewise domain.

To show that $\cC$ is non-noetherian, we modify the proof of Lemma~\ref{lem.Bnnoeth} with slightly less detail.

It suffices to show that $e_0 \cC$ is not noetherian as a right $\cC$-module. By Lemma~\ref{lem.Cbasis}, $e_0 \cC$ has a $\kk$-basis of the form \eqref{eq.CbasisWithXi} with $i=0$.

Since $\beta = 0$, the relations in Definition~\ref{defn.qDU_C} yield
\begin{equation}\label{eq.ACkill}
(\alpha a_i b_{i+1} - x_{i})a_{i+2} 
    = b_i(\alpha a_{i+1} b_{i+2}- x_{i+1}) = 0.
\end{equation}

Set $U= a_0 a_1 \cdots a_{n-1}$. 
For each $s \geq 1$, define
\[ I_s = \sum_{m=1}^s U^m(\alpha a_0 b_1 - x_0)\cC.\]
By \eqref{eq.ACkill}, 
\[ I_s = \sum_{m=1}^s \sum_{j,k \geq 0} \kk U^m (\alpha a_0 b_1 - x_0) x_2 x_4 \cdots x_{2j} b_{2j+2}b_{2j+3}\cdots b_{2j+k+1},\]
with the convention that when $j=0$, $x_2 x_4 \cdots x_{2j} = 1$ and when $k=0$, $b_{2j+2}b_{2j+3}\cdots b_{2j+k+1} = 1$.

By \eqref{eq.ACkill}, $b_i x_{i+1} = b_i (\alpha a_{i+1} b_{i+2}) = x_i(\alpha b_{i+2})$ and hence 
\[
b_1 (x_2 x_4 \cdots x_{2j}) = \alpha^j (x_1 x_3 \cdots x_{2j-1}) b_{2j+1}
\] for any $j \geq 0$.
Thus $I_s$ is contained in the $\kk$-span of the basis monomials
\[
U^m x_0 x_2 x_4 \cdots x_{2j} b_{2j+2}b_{2j+3}\cdots b_{2j+k+1}, \qquad 
U^m a_0 x_1 x_3 \cdots x_{2j-1} b_{2j+1} \cdots b_{2j+k+1}
\]
with $k, j \geq 0$ and $1 \leq m \leq s$. Thus no polynomial in $I_s$ has $U^{s+1} x_{0}$ in its support, while $I_{s+1}$ does contain the polynomial $U^{s+1}(\alpha a_0 b_1 - x_{0})$.

Thus $I_s \subsetneq I_{s+1}$ for all $s \geq 1$ and so $e_0 \cC$ is not noetherian. So $\cC$ is not right noetherian. A similar proof shows $\cC$ is not left noetherian, or we may appeal to $\cC(\alpha, \beta)^\op \cong \cC(\alpha, \beta)$ via Proposition~\ref{prop.Cop}.
\end{proof}

Fix $p,q \in \kk^\times$. Let $C=C_n(q)$ and let $\phi$ be the automorphism of $\kk\hatQ$ defined by 
\[
\phi(e_i) = e_{i-2}, \quad
\phi(a_i) = p a_{i-2}, \quad
\phi(b_i) = p\inv b_{i-2}.
\]
Let $\cD=\cD(\hatQ,\omega,\phi)$ with
\[ \omega = \sum_{i \in Q_0} b_ia_{i+1} - q a_i b_{i+1}.\]
Then the relations in $\cD$ are
\begin{align*}
\omega a_i - \phi(a_i)\omega
%	&= (b_{i-2}a_{i-1} - q a_{i-2} b_{i-1})a_i - pa_{i-2}(b_{i-1}a_i - q a_{i-1} b_i) \\
	&= b_{i-2}a_{i-1}a_i - (p+q) a_{i-2} b_{i-1}a_i + q pa_{i-2}a_{i-1} b_i \\
\omega b_i - \phi(b_i)\omega
%	&= (b_{i-2}a_{i-1} - q a_{i-2} b_{i-1}) b_i - p\inv b_{i-2}(b_{i-1}a_i - q a_{i-1} b_i) \\
	&= -p\inv (b_{i-2}b_{i-1}a_i  - (p+q) b_{i-2}a_{i-1}b_i + pq a_{i-2} b_{i-1}b_i).
\end{align*}
These are precisely the relations given in Definition \ref{defn.qDU_C} with $\alpha=q+p$ and $\beta=-qp$. That is, $\omega$ is a normal element in $\cC=\cC_n(\alpha,\beta)$ and $C=\cC/(\omega)$. Given $\alpha,\beta \in \kk$, the $p,q$ that give $\alpha,\beta$ in the above construction are precisely the roots of the polynomial $t^2-\alpha t - \beta$.

\begin{theorem}\label{thm.typeC}
Fix $n \in \ZZ_+$ and let $\alpha,\beta \in \kk$. Let $\cC=\cC_n(\alpha,\beta)$. The following are equivalent:
\begin{enumerate}
\item $\beta \neq 0$,
\item \label{Cpwd} $\cC$ is a piecewise domain, and
\item \label{Cnoeth} $\cC$ is right (or left) noetherian.
\end{enumerate}
Moreover, if any of the above conditions hold, then $\cC$ is twisted graded Calabi--Yau of dimension three. 
\end{theorem}
\begin{proof}
Suppose $\beta \neq 0$. Set $C=\cC/(\omega)$. By Lemma \ref{lem.Cbasis}, $h_{\cC}^{\tot}=(1-t^2)h_C^{\tot}$ and so, by Lemma \ref{lem.hilb}, $\omega$ is regular. Now \eqref{Cpwd} follows from Lemma \ref{lem.PWD2} and \eqref{Cnoeth} from Proposition \ref{prop.norm_ext}. The converse now follows from Lemma \ref{lem.Cnnoeth}. Finally, given the equivalent conditions, $\cC$ is twisted graded Calabi--Yau of dimension three by Proposition \ref{prop.lift}.
\end{proof}

\begin{corollary}
Fix $\alpha \in \kk$ and $\beta \in \kk^\times$. Let $\cC=\cC_n(\alpha,\beta)$. The Nakayama automorphism $\mu^{\cC}$ is given by
\[ \mu^C(e_i)=e_{i+4}, \quad \mu^{\cC}(a_i)=\beta\inv a_{i+4}, \quad \mu^{\cC}(b_i)=\beta b_{i+4}.\]
The matrix-valued Hilbert series of $\cC$ is $h_{\cC}(t)=(I-Mt+P^Ct^2)\inv(I-P^Ct^2)\inv$ where $P^C$ is defined as in \eqref{eq.PtypeC}.
\end{corollary}

\section{Normal extensions through Ore extensions}
\label{sec.ore}

In this section, we establish that the normal extensions constructed above may be realized through Ore extensions. Our strategy is inspired in part by the strategy of Minamoto for quiver Heisenberg algebras \cite{min}. However, we do not have the restriction that the quiver is acyclic. 

Let $R$ be a $\kk$-algebra and $\sigma$ an automorphism of $R$. A \emph{$\sigma$-derivation} of $R$ is a $\kk$-linear map $\delta:R \to R$ satisfying
\[ \delta(rr') = \sigma(r)\delta(r') + \delta(r)r' \quad\text{for all $r,r' \in R$.}\]
Given the above, the \emph{Ore extension} $S=R[x;\sigma,\delta]$ is generated over $R$ by $x$ subject to the relations $xr = \sigma(r)x + \delta(r)$ for all $r \in R$. If $R$ is twisted Calabi--Yau of dimension $d$, then $S$ is twisted Calabi--Yau of dimension $d+1$ by \cite[Theorem 3.3]{LWW2}.

Suppose that $R=\kk Q/I$ with $I \subset \kk Q_{\geq 2}$ homogeneous. Further, suppose that $\sigma$ is a graded automorphism and $\delta$ is homogeneous of degree one. This implies that $\sigma$ permutes the set of idempotents, so $\sigma(e_i)=e_{\sigma(i)}$ and $\sigma$ commutes with the source and target maps ($s$ and $t$, respectively). That is, if $a \in Q_1$ with $s(a)=e_i$ and $t(a)=e_j$, then $\sigma(a)$ satisfies $s(\sigma(a))=e_{\sigma(i)}$ and $t(\sigma(a))=e_{\sigma(j)}$. Moreover, $\delta(a) \in e_{\sigma(i)} R_2 e_{j}$. 

In $S=R[x;\sigma,\delta]$, we have $x = \sum x_i$ where $x_i = xe_i = e_{\sigma(i)}x$. That is, $x_i$ has $s(x_i) = e_{\sigma(i)}$ and $t(x_i)=e_i$. Then for $a \in e_i R e_j$, $x_i a = \sigma(a)x_j + \delta(a)$ is a homogeneous relation from $e_{\sigma(i)}$ to $e_j$. 

\begin{lemma}\label{lem.regular}
Let $R$ be a ring with $z \in R$ regular. Let $\sigma$ be an automorphism of $R$ and $\delta$ be a $\sigma$-derivation of $R$. Then $z$ is regular in $S=R[x;\sigma,\delta]$.
\end{lemma}

\begin{proof}
Suppose for contradiction that $z$ is not regular. Then there exists $f \in S \setminus \{0\}$ such that $zf=0$ or $fz=0$. Let $r \in R \setminus \{0\}$ and $d \in \NN$ such that $rx^d$ is the term of $f$ of highest $x$-degree. Then
\[ zf = zrx^d + \text{lower $x$-degree terms}.\]
So if $zf=0$, then $zr = 0$. But $z$ is regular in $R$, so $zr \neq 0$, a contradiction.

Similarly,
\[ fz = rx^dz + \text{lower $x$-degree terms}.\]
Then $rx^dz = r\sigma^d(z) x^d + $ lower $x$-degree terms. Thus, if $fz=0$, then $r\sigma^d(z) = 0$. But $\sigma$ is an $R$-automorphism, so $\sigma^d(z)$ is also regular in $R$. Hence, $r\sigma^d(z) \neq 0$ and we again have a contradiction.

Thus $z$ is regular in $S$.
\end{proof}

Let $Q$ be as in \eqref{defn.typeA}. We will build all of our algebras as iterated Ore extensions of the path algebra $\kk Q$, which is Calabi--Yau of dimension one with Nakayama automorphism satisfying $\mu(e_i)=e_{i+1}$ \cite[Proposition 6.6]{RR1}. Let $\sigma$ be a graded automorphism of $\kk Q$. Hence, there are scalars $q_i \in \kk^\times$ such that $\sigma(e_i) = e_{\sigma(i)}$ and $\sigma(a_i)=q_i a_{\sigma(i)}$ for all $i$.

Set $A=\kk Q[x;\sigma]$ so that $A$ is twisted Calabi--Yau of dimension two. Set $x = \sum x_i$ where $x_i = xe_i = e_{\sigma(i)} x$. Then the relations on $A$ are $x_i a_i = q_i a_{\sigma(i)}x_{i+1}$ for $i\in Q_0$. Note that this is a homogeneous relation $e_{\sigma(i)} \to e_{i+1}$ and so $\mu^A(e_i) = e_{\sigma\inv(i)+1}$. Let $\fQ$ be the underlying quiver of $A$ so that $A = \kk\fQ/(\omega)$ where
\[ \omega = \sum_{i \in Q_0} x_ia_i - q_i a_{\sigma(i)}x_{i+1}.\]

\begin{theorem}\label{thm.ore}
Let $A$ be as above and suppose that $q_i = q_{\sigma(i-1)}$ for all $i$. Define $\phi \in \Aut_{\gr}(A)$ by
\[ 
\phi(e_i) = e_{\sigma(i-1)}, \quad
\phi(a_i) = p_i a_{\sigma(i-1)}, \quad
\phi(x_i) = p_i\inv x_{\sigma(i-1)},
\] 
where $p_i \in \kk^\times$ and $p_{\sigma(i)}=p_{i+1}$ for all $i\in Q_0$. Then the normal extension $\cD(\fQ,\omega,\phi)$ is twisted Calabi--Yau of dimension three.
\end{theorem}
\begin{proof}
Define $R=\kk Q[z;\phi]$ by $\phi(e_i) = e_{\sigma(i-1)}$ and $\phi(a_i) = p_i a_{\sigma(i-1)}$. Then $z = \sum z_i$ where $z_i = ze_i = e_{\sigma(i-1)}z$. So, the relations on $R$ are $z_i a_i = p_ia_{\sigma(i-1)}z_{i+1}$.

We now extend the automorphism $\sigma$ of $\kk Q$ to $R$ by setting $\sigma(z_i)=p_i z_{\sigma(i)}$. This is then an automorphism of $R$ because:
\begin{align*}
\sigma(z_i)\sigma(a_i) -    
    p_i\sigma(a_{\sigma(i-1)})\sigma(z_{i+1})
    &= p_i q_i z_{\sigma(i)} a_{\sigma(i)} - p_i p_{i+1} q_{\sigma(i-1)}a_{\sigma^2(i-1)}z_{\sigma(i+1)} \\
    &= p_i q_i \left(z_{\sigma(i)} a_{\sigma(i)} - p_{\sigma(i)} a_{\sigma^2(i-1)}z_{\sigma(i+1)} \right) = 0.
\end{align*}

Now define a $\kk$-linear map on $R$ by:
\[
\delta(e_i) = 0, \qquad
\delta(a_i) = z_{i+1}, \qquad
\delta(z_i) = 0.
\]
We verify that this is a $\sigma$-derivation of $R$ below:
\begin{align*}
    \delta(z_i a_i - p_ia_{\sigma(i-1)}z_{i+1})
        &= \left(\sigma(z_i)\delta(a_i) + \delta(z_i)a_i\right)
            - p_i\left( \sigma(a_{\sigma(i-1)})\delta(z_{i+1}) + \delta(a_{\sigma(i-1)})z_{i+1}\right) \\
        &= p_iz_{\sigma(i)}z_{i+1} - p_i z_{\sigma(i)}z_{i+1} = 0.
%        &= p_i\left(z_{\sigma(i)}z_{i+1} - z_{\sigma(i)}z_{i+1}\right)= 0.
\end{align*}

Set $S=R[x;\sigma,\delta]$. Then $S$ is twisted Calabi--Yau of dimension three and the relations on $S$ are
\begin{align*}
    xa_i &= x_ia_i = q_i a_{\sigma(i)}x_{i+1} + \delta(a_i) 
        = q_ia_{\sigma(i)}x_{i+1} + z_{i+1} \\
    xz_i &= x_{\sigma(i-1)}z_i 
        = p_ix_{\sigma(i)}y_i.
\end{align*}
Note that $z_{i+1} = ze_{i+1} = e_{\sigma(i+1)}z$. So, we may replace each $z_{i+1}$ with the corresponding relation $\omega_i = x_ia_i - q_i a_{\sigma(i)}x_{i+1}$. The element $z$ is regular by Lemma \ref{lem.regular} and so, setting $x=a^*$ gives that $H$ is a normal extension of $A$, so $S \iso \cD(\fQ,\omega,\phi)$.
\end{proof}

\begin{corollary}
Let $\cD=\cD(\fQ,\omega,\phi)$ be as in Theorem \ref{thm.ore}. Then $\cD$ is a noetherian piecewise domain of GK dimension three.
\end{corollary}
\begin{proof}
That $\cD$ is a piecewise domain follows from the construction of $S$ and \cite[Proposition 2.3]{GK1}. The algebra $A$ is twisted Calabi--Yau of dimension two and GK dimension two. Since $\cD/(\omega) \iso A$, then $\cD$ is noetherian by \cite[Lemma 8.2]{ATV1}.

Because $A$ is twisted Calabi--Yau of dimension two, then the adjacency matrix $M$ of $\fQ$ is normal \cite[Proposition 7.1]{RR1}. Now the GK dimension of $\cD$ is three by \cite[Proposition 8.10]{RR1}.
\end{proof}

We now show how our three families of twisted graded Calabi--Yau algebras fit into the paradigm of Theorem \ref{thm.ore}.

\begin{proposition}\label{prop.ore}
Fix $n \geq 2$ and $p,q \in \kk^\times$. 
    
    (Type A) Let $\dblQ$ be the quiver from \eqref{eq.dblA}. 
    Define $\phi \in \kk\dblQ$ by
    \[
        \phi(e_i) = e_i, \quad
        \phi(a_i) = p a_i, \quad
        \phi(a_i^*) = p\inv a_i^*.
    \]
    Then the normal extension $\cD=\cD(\dblQ,\omega,\phi)$ is twisted Calabi--Yau of dimension three of type $(M,I,4)$ with Hilbert series $h_{\cD}(t)=(I-Mt+It^2)\inv(I-It^2)\inv$. Moreover, if we set $\alpha_i = q + p$ and $\beta_i = -pq$ for $i\in\dblQ_0$, then $\cD \iso \cH_n(\alpha,\beta)$.

    (Type B) Let $\tldQ$ be the quiver from \eqref{eq.Bn}.
    Define $\phi \in \kk\tldQ$ by
    \[
        \phi(e_i) = e_{i-1}, \quad
        \phi(a_i) = p a_{i-1}, \quad
        \phi(b_i) = p\inv b_{i-1}.
    \]
    Then the normal extension $\cD(\tldQ,\omega,\phi)$ is twisted Calabi--Yau of dimension three and type $(M,P^2,4)$ with Hilbert series $h_{\cD}(t)=(I-Mt+Pt^2)\inv(I-Pt^2)\inv$. Moreover, if we set $\alpha = q+p$ and $\beta = -pq$, then $\cD \iso \cB_n(\alpha,\beta)$.

    (Type C) Let $\hatQ$ be the quiver from \eqref{eq.typeC}. 
    Define $\phi \in \kk\hatQ$ by
    \[
        \phi(e_i) = e_{i-2}, \quad
        \phi(a_i) = p a_{i-2}, \quad
        \phi(b_i) = p\inv b_{i-2}.
    \]
    Then the normal extension $\cD(\tldQ,\omega,\phi)$ is twisted Calabi--Yau of dimension three and type $(M,P^2,4)$ with Hilbert series $h_{\cD}(t)=(I-Mt+Pt^2)\inv(I-Pt^2)\inv$. Moreover, if we set $\alpha = q+p$ and $\beta = -pq$, then $\cD \iso \cC_n(\alpha,\beta)$.

\end{proposition}
\begin{proof}
In each case, it suffices to realize the corresponding Calabi--Yau algebra of dimension two as an Ore extension of $\kk Q$ with $Q$ as in \eqref{eq.typeA}. 
\begin{enumerate}
    \item[(Type A)] Let $\sigma_A \in \Aut_{\gr}(\kk Q)$ be defined by $\sigma_A(e_i)=e_{i+1}$ and $\sigma_A(a_i)=q a_{i+1}$. Then $\kk Q[x;\sigma_A] \iso A_n(\bq)$ via the isomorphism $a_i \mapsto a_i$ and $x_i \mapsto a_i^*$. 

    \item[(Type B)] Let $\sigma_B \in \Aut_{\gr}(\kk Q)$ be defined by $\sigma_B(e_i)=e_i$ and $\sigma_B(a_i)=q a_i$. Then $\kk Q[x;\sigma_B] \iso B_n(\bq)$ via the isomorphism $a_i \mapsto a_i$ and $x_i \mapsto b_i$. 

    \item[(Type C)] Let $\sigma_C \in \Aut_{\gr}(\kk Q)$ be defined by $\sigma_C(e_i)=e_{i-1}$ and $\sigma_C(a_i)=q a_{i-1}$. Then $\kk Q[x;\sigma_C] \iso C_n(\bq)$ via the isomorphism $a_i \mapsto a_i$ and $x_i \mapsto b_{i-1}$. 
\end{enumerate}
Because we have restricted to the single parameter case, the conditions $q_i=q_{\sigma(i-1)}$ and $p_{\sigma(i)}=p_{i+1}$ are satisfied trivially. Then we may apply Theorem \ref{thm.ore} in each case.
\end{proof}

\section{Dimension four twisted graded Calabi--Yau algebras}
\label{sec.dim4}

Let $A$ be an Artin--Schelter regular algebra of (global) dimension 4 and generated in degree 1. Assume $A$ is a domain or $\GKdim(A) \geq 3$. By \cite[Proposition 1.4]{LPWZ}, $A$ is one of the following types:
\begin{enumerate}
    \item \label{d4-quad} $A$ is generated by four elements with six relations of degree 2,
    \item \label{d4-qc} $A$ is generated by three elements with two relations in degree 2 and two relations in degree 3, or
    \item \label{d4-ne} $A$ is generated by two elements with one relation of degree 3 and one relation of degree 4.
\end{enumerate}
Each of these corresponds to a specific minimal resolution of the trivial module $\kk_A$, as well as a specific Hilbert series. All of these types may be obtained as normal extensions of an Artin--Schelter algebra of dimension 3. Here we present examples of dimension four twisted graded Calabi--Yau algebras of dimension four that may also be obtained through normal extension, and which roughly correspond to each category. 

\subsection{Ore extensions}

Let $A=\kk Q/I$ be a twisted graded Calabi--Yau algebra of dimension 3.
A polynomial extension of $A$, which is trivially a normal extension, is a twisted graded Calabi--Yau algebra of dimension 4. More generally, one can take Ore extensions of the type $A[x;\sigma]$. 

Suppose $M$ is the adjacency matrix of $Q$ and $P$ the permutation matrix corresponding to $\mu_0^A$, so that $A$ has Hilbert series
\[    h_A(t) = (I-Mt + M^T t^{s-1} - Pt^s)\inv\]
where $s=3$ or $4$ \cite[Proposition~8.2]{RR1} (see also \cite[p.~255]{RR1}). Let $\sigma$ be a graded automorphism of $A$ and let $P_\sigma$ be the permutation matrix corresponding to $\sigma$ (so $(P_\sigma)_{ij} = \delta_{\sigma(i)j})$. Then $B=A[x;\sigma]$ has Hilbert series
\[ h_B(t) = (I-Mt + M^T t^{s-1} - Pt^s)\inv(I-P_\sigma\inv t)\inv. \]
Hence,
\[ 
    h_B = \begin{cases}
        (I - (M+P_\sigma\inv)t^2 + (MP_\sigma\inv + M^T)t^2 - (M^T P_\sigma\inv + P)t^3 + PP_\sigma\inv t^4)\inv & s=3 \\
        (I - (M+P_\sigma\inv)t^2 + MP_\sigma\inv t^2 + M^T t^3 - (M^T P_\sigma\inv + P)t^4 + PP_\sigma\inv t^5)\inv & s=4.
    \end{cases}
\]
The first corresponds to type \eqref{d4-quad} above in the sense that all relations will be quadratic. The second corresponds to type \eqref{d4-qc} as there will be quadratic and cubic relations. 

\subsection{Homogenized down-up algebras}
\label{sec.homog}

Here we give an alternate way to produce dimension four twisted graded Calabi--Yau algebras of type \eqref{d4-qc}.

When $\gamma \neq 0$, the relations for the quiver down-up algebra $\cH_n(\alpha,\beta,\gamma)$ can be homogenized resulting in a twisted Calabi--Yau algebra of dimension four. In the connected case, this was considered by Cassidy \cite{cass-homog}.

Let $A$ be a filtered algebra with filtration $\cF=\{F_n(A)\}$. The \emph{Rees ring} corresponding to the filtration $\cF$ is the subring of $A[z]$ defined as $\cR_{\cF}(A) = \sum F_n(A) z^n$. As with the associated graded algebra, we drop the subscript if the filtration is understood.

\begin{lemma}\label{lem.rees-pwd}
If $A$ is a filtered piecewise domain, then $\cR(A)$ is a piecewise domain.
\end{lemma}
\begin{proof}
This is immediate as $\cR(A)$ is a subring of $A[z]$.
\end{proof}

\begin{definition}\label{defn.hqdu}
Fix $n \in \ZZ_+$ and let $\alpha,\beta,\gamma \in \kk^n$. The corresponding \emph{homogenized quiver down-up algebra} $\cR_n(\alpha,\beta,\gamma)$ is the Rees ring of $\cH_n(\alpha,\beta,\gamma)$ corresponding to the standard filtration.
\end{definition}

Explicitly, let $Q$ be the following quiver: 
\begin{equation}\label{eq.homogA}
\begin{tikzcd}
    &         &   	& e_0 \arrow[llld, "u_0"] \arrow[loop above,looseness=10, "z_0"] \arrow[rrrd, "d_{n-1}", shift left=2] 	&           &		&          \\
e_1 \arrow[rr, "u_1"] \arrow[rrru, "d_0", shift left=2] \arrow[loop left,looseness=10, "z_1"] 	&        	&                   	e_2 \arrow[r, "u_2"]	\arrow[ll, "d_1", shift left=2]		\arrow[loop below,looseness=10, "z_2"]	 		& \cdots   \arrow[r, "u_{n-3}"]  \arrow[l, "d_2", shift left=2] &	e_{n-2} \arrow[rr, "u_{n-2}"]	\arrow[l, "d_{n-3}", shift left=2] \arrow[loop below,looseness=10, "z_{n-2}"] &           	& e_{n-1} \arrow[lllu, "u_{n-1}"] \arrow[ll, "d_{n-2}", shift left=2] \arrow[loop right,looseness=10, "z_{n-1}"]
\end{tikzcd}
\end{equation}
Then $\cR(\alpha,\beta,\gamma)$ is the quotient of $\kk Q$ by the relations
\begin{align}\label{eq.typeHA}
	d_{i-1}u_{i-1}u_i &= \alpha_i u_id_iu_i + \beta_i u_iu_{i+1}d_{i+1} +\gamma_i u_i z_{i+1}^2 \\
	d_id_{i-1}u_{i-1} &= \alpha_i d_iu_id_i + \beta_i u_{i+1}d_{i+1}d_i +\gamma_i d_i z_i^2 \\
    z_i u_i &= u_i z_{i+1} \qquad z_{i+1} d_i = d_i z_i,
\end{align}
for $i\in Q_0$.

\begin{lemma}
Fix $n \in \ZZ_+$ and let $\alpha,\beta,\gamma \in \kk^n$ with $\beta_i \neq 0$ for all $i$. Then $\cR=\cR_n(\alpha,\beta,\gamma)$ is twisted graded Calabi--Yau of dimension four with matrix Hilbert series 
\[h_{\cR}(t) = (I-Mt+M^Tt^3+Pt^4)\inv (I-It)\inv.\]
\end{lemma}
\begin{proof}
By Lemma \ref{lem.rees-pwd}, $\cR$ is a piecewise domain. Hence, $z=\sum z_i$ is a regular element and it is clear that $\cR/(z) \iso \cH_n(\alpha,\beta)$. Then $\cR$ is twisted Calabi--Yau by Proposition \ref{prop.lift} and the Hilbert series is due to Lemma \ref{lem.hilb}.
\end{proof}

\subsection{Normal extensions of quiver down-up algebras}
\label{sec.extdu}

A way to realize Artin--Schelter regular algebras of type \eqref{d4-ne} is to take normal extensions of down-up algebras. We follow a similar strategy applied to a special family of quiver down-up algebras.

Fix $n \in \ZZ_+$ and set $\cH=\cH_n(0,-1)$. Define
\begin{align}\label{eq.omega4} 
\omega = \sum_{i \in \dblQ_0} d_{i-1}u_{i-1}u_i + u_iu_{i+1}d_{i+1}.
\end{align}

Define $\phi \in \Aut(\kk\dblQ)$ by
\[
\phi(e_i) = e_{i-1}, \quad
\phi(u_i) = u_{i-1}, \quad
\phi(d_i) = d_{i-1},
\]
for all $i \in \dblQ_0$. Now in $\kk\dblQ$ we have
\begin{align*}
0 = \omega u_i - \phi(u_i) \omega
    &= \omega u_i - u_{i-1} \omega
    = ( d_{i-2}u_{i-2}u_{i-1}u_i+ u_{i-1}u_id_iu_i ) - 
        ( u_{i-1}d_{i-1}u_{i-1}u_i + u_{i-1}u_iu_{i+1}d_{i+1}) \\
    &= d_{i-2}u_{i-2}u_{i-1}u_i - u_{i-1}d_{i-1}u_{i-1}u_i
        + u_{i-1}u_id_iu_i - u_{i-1}u_iu_{i+1}d_{i+1} \\
0 = \omega d_i - \phi(d_i) \omega
    &= \omega d_i - d_{i-1} \omega
    = (d_{i-1}u_{i-1}u_id_i + u_iu_{i+1}d_{i+1}d_i)
        - (d_{i-1}d_{i-2}u_{i-2}u_{i-1} + d_{i-1}u_{i-1}u_id_i) \\
    &= -( d_{i-1}d_{i-2}u_{i-2}u_{i-1} - u_iu_{i+1}d_{i+1}d_i).
\end{align*}
We also keep the other relation type from $\cH$, so the relation $0=\omega d_i- \phi(d_i)\omega$ is superfluous. Hence, in this case, $\cD$ is the quotient of $\kk\dblQ$ by the relations
\begin{equation}\label{eq.Drelns}
\begin{aligned}
    d_id_{i-1}u_{i-1} &= -u_{i+1}d_{i+1}d_i \\
    d_iu_iu_{i+1}u_{i+2} &= u_{i+1}d_{i+1}u_{i+1}u_{i+2}
        - u_{i+1}u_{i+2}d_{i+2}u_{i+2} + u_{i+1}u_{i+2}u_{i+3}d_{i+3}.
\end{aligned}
\end{equation}

In the next lemma, we interpret the empty product (say from $j=0$ to $-1$) as $1$.

\begin{lemma}
Fix $n \in \ZZ_+$ and let $\cD$ be as defined above. A $\kk$-basis for $\cD$ consists of paths of the form:
\begin{align}\label{eq.Dbasis}
\left( \prod_{j=0}^{k_1-1} u_{i+j} \right)
\left( \prod_{j=0}^{k_2-1} d_{i+(k_1-1)+j}u_{i+(k_1-1)+j}u_{i+k_1+j} \right)
\left( d_{i+k_1+(k_2-1)}u_{i+k_1+(k_2-1)} \right)^{k_3}
\left( \prod_{j=0}^{k_4-1} d_{i+k_1+(k_2-1)-j} \right)
\end{align}
with $k_1,k_2,k_3,k_4 \in \ZZ_{\geq 0}$.  Consequently, $h_{\cD}^{\tot}(t) = n(1-t)^{-2}(1-t^2)\inv(1-t^3)\inv$.
\end{lemma}
\begin{proof}
We order the arrows $d_0 > d_1 > \cdots > d_{n-1} > u_0 > u_1 > \cdots > u_{n-1}$. The leading terms under this ordering are $d_id_{i-1}u_{i-1}$ and $d_iu_iu_{i+1}u_{i+2}$. This results in one type of ambiguity, $d_id_{i-1}u_{i-1}u_iu_{i+1}$, which we attempt to resolve:
\begin{align*}
d_i&(d_{i-1}u_{i-1}u_iu_{i+1}) - (d_id_{i-1}u_{i-1})u_iu_{i+1} \\
    &= d_i(u_{i}d_{i}u_{i}u_{i+1} - u_{i}u_{i+1}d_{i+1}u_{i+1} + u_{i}u_{i+1}u_{i+2}d_{i+2}) - (-u_{i+1}u_{i+2}u_{i+3}d_{i+3}d_{i+2}) \\
    &= d_iu_{i}d_{i}u_{i}u_{i+1} - d_iu_{i}u_{i+1}d_{i+1}u_{i+1} 
    + (u_{i+1}d_{i+1}u_{i+1}u_{i+2} - u_{i+1}u_{i+2}d_{i+2}u_{i+2} \\
    &\qquad + u_{i+1}u_{i+2}u_{i+3}d_{i+3})d_{i+2} 
    + u_{i+1}u_{i+2}u_{i+3}d_{i+3}d_{i+2} \\
    &= d_iu_{i}d_{i}u_{i}u_{i+1} - d_iu_{i}u_{i+1}d_{i+1}u_{i+1} 
    + u_{i+1}d_{i+1}u_{i+1}u_{i+2}d_{i+2} - u_{i+1}u_{i+2}d_{i+2}u_{i+2}d_{i+2} \\
    &\qquad + u_{i+1}u_{i+2}u_{i+3}d_{i+3}d_{i+2} 
    + u_{i+1}u_{i+2}u_{i+3}d_{i+3}d_{i+2}.
\end{align*}
This gives a new relation:
\[
d_iu_{i}d_{i}u_{i}u_{i+1} = d_iu_{i}u_{i+1}d_{i+1}u_{i+1} 
    - u_{i+1}d_{i+1}u_{i+1}u_{i+2}d_{i+2} + u_{i+1}u_{i+2}d_{i+2}u_{i+2}d_{i+2}
    - 2u_{i+1}u_{i+2}u_{i+3}d_{i+3}d_{i+2}.
\]
Hence, our Grobner basis has three leading terms: $d_id_{i-1}u_{i-1}$, $d_iu_iu_{i+1}u_{i+2}$, and $d_iu_{i}d_{i}u_{i}u_{i+1}$. The two (new) ambiguities are $d_id_{i-1}u_{i-1}d_{i-1}u_{i-1}u_i$ and $d_iu_id_iu_iu_{i+1}u_{i+2}$, and these resolve through tedious computation that we omit. 

Hence, a $\kk$-basis consists of paths that avoid the leading terms. It is clear that these are the basis in \eqref{eq.Dbasis}. The claim on the total Hilbert series follows immediately.
\end{proof}

\begin{proposition}
The algebra $\cD$ defined above is twisted graded Calabi--Yau of dimension four. 
\end{proposition}
\begin{proof}
This follows as before since $\cD/(\omega) \iso \cH$ with $\omega$ as in \eqref{eq.omega4} and $\cD$ has the correct Hilbert series.
\end{proof}

\end{document}